\documentclass[a4paper,10pt,leqno]{amsart}
        \title {K- and L-theory of graph products of groups}
       \author{Daniel Kasprowski}
        \address{Mathematisches Institut der Universit\"at Bonn\\
                Endenicher Allee 60\\
                53115 Bonn, Germany}
         \email{kasprowski@uni-bonn.de}
         \urladdr{http://www.math.uni-bonn.de/people/daniel}
         \author{Kevin Li}
        \address{Mathematisches Institut der Universit\"at Bonn\\
                Endenicher Allee 60\\
                53115 Bonn, Germany}
         \email{s6keliii@uni-bonn.de}
                  \author{Wolfgang L\"uck}
        \address{Mathematisches Institut der Universit\"at Bonn\\
                Endenicher Allee 60\\
                53115 Bonn, Germany}
         \email{wolfgang.lueck@him.uni-bonn.de}
          \urladdr{http://www.him.uni-bonn.de/lueck}
         \date{\today}

\usepackage[colorlinks]{hyperref}
\usepackage{color}
\usepackage{pdfsync}
\usepackage{calc}
\usepackage{enumerate,amssymb}
\usepackage[arrow,curve,matrix,tips,2cell]{xy}
  \SelectTips{eu}{10} \UseTips
  \UseAllTwocells
\usepackage{mathtools}

\DeclareMathAlphabet{\matheurm}{U}{eur}{m}{n}

\newcommand{\Groupoids}{\matheurm{Groupoids}}
\newcommand{\Groups}{\matheurm{Groups}}

\newcommand{\Spectra}{\matheurm{Spectra}}
\newcommand{\Modcat}[1]{#1\text{-}\matheurm{Modules}}

\newcommand{\eub}[1]{\underline{E}#1}              
     
\newcommand{\edub}[1]{\underline{\underline{E}}#1} 

\DeclareMathOperator{\ad}{ad}

\DeclareMathOperator{\ch}{ch}

\DeclareMathOperator{\cok}{cok}
\DeclareMathOperator{\colim}{colim}

\DeclareMathOperator{\cone}{cone}

\DeclareMathOperator{\id}{id}
\DeclareMathOperator{\im}{im}

\DeclareMathOperator{\ind}{ind}

\DeclareMathOperator{\mor}{mor}

\DeclareMathOperator{\pt}{pt}
\DeclareMathOperator{\pr}{pr}
\DeclareMathOperator{\RES}{RES}

\DeclareMathOperator{\Tor}{Tor}

\DeclareMathOperator{\Wh}{Wh}

\newcommand{\calfin}{{\mathcal F}{\mathcal I}{\mathcal N}}

\newcommand{\calvcyc}{{\mathcal V}{\mathcal C}{\mathcal Y}}

  \newcommand{\IC}{\mathbb{C}}

  \newcommand{\IQ}{\mathbb{Q}}
  \newcommand{\IR}{\mathbb{R}}

  \newcommand{\IZ}{\mathbb{Z}}

  \newcommand{\calc}{\mathcal{C}}

  \newcommand{\cale}{\mathcal{E}}
  \newcommand{\calf}{\mathcal{F}}
   
  \newcommand{\calg}{\mathcal{G}}
  \newcommand{\calh}{\mathcal{H}}

  \newcommand{\calk}{\mathcal{K}}

  \newcommand{\calp}{\mathcal{P}}

  \newcommand{\cals}{\mathcal{S}}

  \newcommand{\calw}{\mathcal{W}}

  \newcommand{\bfE}{{\mathbf E}}
  \newcommand{\bff}{{\mathbf f}}

  \newcommand{\bfK}{{\mathbf K}}
  \newcommand{\bfL}{{\mathbf L}}

\newcommand{\EGF}[2]{E_{#2}(#1)}

\newcommand{\SubGF}[2]{\matheurm{Sub}_{#2}(#1)}

\newcounter{commentcounter}

\theoremstyle{plain}
\newtheorem{theorem}{Theorem}[section]

\newtheorem{lemma}[theorem]{Lemma}

\newtheorem*{theorem*}{Theorem}
\newtheorem*{theoremA*}{Theorem A}
\newtheorem*{theoremB*}{Theorem B}

\theoremstyle{definition}
\newtheorem{definition}[theorem]{Definition}

\newtheorem{example}[theorem]{Example}

\newtheorem{remark}[theorem]{Remark}

\newtheorem*{definition*}{Definition}

\theoremstyle{remark}

\makeatletter\let\c@equation=\c@theorem\makeatother

\newcommand{\version}[1]              
{\begin{center} last edited on #1\\
last compiled on \today\\
name of tex-file: \jobname
\end{center}
}

\begin{document}

\begin{abstract}
  We compute the group homology, the algebraic $K$- and $L$-groups, and the topological $K$-groups of
  right-angled Artin groups, right-angled Coxeter groups, and more generally, graph products.
\end{abstract}

\maketitle

\newlength{\origlabelwidth} \setlength\origlabelwidth\labelwidth


\typeout{------------------- Introduction -----------------}
\section{Introduction}
\label{sec:introduction}


\subsection{Basic setup}
\label{subsec:Basic_setup}

   Suppose that we are given the following data:

   \begin{itemize}

   \item A finite simplicial graph $X$ on the vertex set $V$ and a collection of groups $\calw \coloneqq \{W_v \mid v \in V\}$. Denote by
     $W = W(X,\calw)$ the associated graph product, see Section~\ref{sec:Graph_product_of_groups}, and
     by $\Sigma$ the flag complex associated to $X$.  Examples of graph products are right-angled Artin
     groups and right-angled Coxeter groups;

   \item A commutative ring with unit $\Lambda$ and an equivariant homology theory
     $\calh^?_*$ with values in $\Lambda$-modules, see
     Definition~\ref{def:equivariant_homology_theory}.  Our main examples will be those associated to algebraic $K$- and $L$-theory or topological $K$-theory, which appear in the Farrell--Jones Conjecture or the Baum--Connes Conjecture;

   \item A non-empty class $\calc$ of finite groups which is closed under isomorphisms,
     passage to subgroups and passage to quotient groups. Our main example will be the
     class of all finite groups;

   \item A class $\cale$ of $\Lambda$-modules with the property that for an exact sequence
     $0 \to V_0 \to V_1 \to V_2 \to 0$ the $\Lambda$-module $V_1$ belongs to $\cale$ if
     and only if both $V_0$ and $V_2$ belong to $\cale$.

\end{itemize}


\subsection{Main result}
\label{subsec:main_result}

Fix an integer $n$. We obtain a covariant functor
\[
\calh^{\calc}_n \colon \Groups \to \Modcat{\Lambda}, \quad G \mapsto \calh_n^G(\EGF{G}{\calc}),
\]  
where $\EGF{G}{\calc}$ is the classifying space of the family of subgroups of $G$ which
belong to $\calc$, see Section~\ref{sec:Equivariant_homology_theories} and~\eqref{calh_upper_calc_n}.

Let $\cals$ be the poset of flag subcomplexes of $\Sigma$ and let
$\calp$ be the poset of simplices of $\Sigma$, both ordered by inclusion,
where the empty subcomplex and the empty simplex are allowed.  For an
element $L$ in $\cals$, we can consider the subgraph $X \cap L$ of
$X$.  Let $W(L)$ be the graph product associated to $X \cap L$ and the
collection of groups
$\calw|_{V\cap L} = \{W_v \mid v \; \text{is a vertex of}\; L\}$. With this
notation $W(\Sigma)$ is the graph product $W(X,\calw)$ and $W(\emptyset) = \{1\}$. We obtain a
covariant functor
\[
W_* \colon \cals \to \Groups, \quad L \mapsto W(L).
\]
Let
\[
I  \colon \calp\to \cals
\]
be the inclusion which sends a simplex $\sigma$ of $\Sigma$ to the corresponding flag subcomplex of $\Sigma$.  Sometimes we identify $\sigma$ in $\calp $ with $I(\sigma)$ in
$\cals$. For instance we will often write $W(\sigma)$ instead of $W(I(\sigma))$.
Notice that the covariant functor $W_* \circ I \colon \calp \to \Groups$ sends a
simplex $\sigma $ of $\Sigma$ to $\prod_{v \in V \cap \sigma} W_v$.

We obtain a covariant functor
\[
  \calh^{\calc}_n \circ W_*  \colon \cals \to \Modcat{\Lambda},
  \quad L \mapsto \calh_n^{W(L)}(\EGF{W(L)}{\calc}).
\]
We are interested in the value at $\Sigma$, i.e., in
$\calh_n^{W}(\EGF{W}{\calc})$ for $W = W(X,\calw)$.

The composite
$\calh^{\calc}_n \circ W_* \circ I \colon \calp \to \Modcat{\Lambda}$ is given by
\[\sigma \mapsto \calh^{\calc}_n\left(\prod_{v \in V \cap \sigma} W_v\right)
  = \calh_n^{\prod_{v \in V\cap \sigma} W_v}\left(\prod_{v \in V \cap \sigma}\EGF{W_v}{\calc}\right),
\]
since $\prod_{v \in V \cap \sigma}\EGF{W_v}{\calc}$ is a model for $\EGF{\prod_{v \in V \cap  \sigma} W_v}{\calc}$.

Define for a simplex $\sigma$ the quotient $\Lambda$-module of
$\calh^{\calc}_n \circ W_* \circ I(\sigma)$ by
\[S_{\sigma} \bigl(\calh^{\calc}_n \circ W_* \circ I\bigr) \coloneqq 
  \cok\left(\bigoplus_{\tau} \calh^{\calc}_n \circ W_*\circ I(\tau) \to
      \calh^{\calc}_n \circ W_*\circ I(\sigma)\right),
    \]
    where $\tau$ runs through the simplices of $I(\sigma)$ which are different from
    $\sigma$.  The idea is to kill everything in
    $\calh^{\calc}_n \circ W_* \circ I(\sigma)$ which comes from a proper
    simplex $\tau$ of $I(\sigma)$.

For a simplex $\sigma$ of $\Sigma$, let $\ch_n(\Sigma,\sigma)$ be the
set of $n$-chains $\sigma_0 < \sigma_1 < \cdots < \sigma_n$ in $\calp$ with
$\sigma_0 = \sigma$. Define the integer
\begin{equation*}
n_{\sigma} \coloneqq  \sum_{n \ge 0} (-1)^n \cdot |\ch_n(\Sigma,\sigma)|.
\end{equation*}

Denote by $G_0(\cale)$ the Grothendieck group of elements in $\cale$, i.e., the abelian
group with the isomorphism classes of elements in $\cale$ as generators and relations
$[V_1] = [V_0]+ [V_2]$ for every short exact sequence $0 \to V_0 \to V_1 \to V_2 \to 0$ of
$\Lambda$-modules belonging to $\cale$.

\begin{theorem}[Main Theorem]\label{the:main_theorem}
  \
  \begin{enumerate}
\item\label{the:main_theorem:iso_T}
 The canoncial $\Lambda$-homomorphism
\[
  T \colon  \colim_{\sigma \in \calp} \calh_n^{W(\sigma)}(\EGF{W(\sigma)}{\calc}) \xrightarrow{\cong}
    \calh_n^{W(\Sigma)}(\EGF{W(\Sigma)}{\calc})
\]
is an isomorphism;

\item\label{the:main_theorem:splitting_the_colimit_iso_T}
  For every $\sigma \in \calp$, the canonical projection
  \[
    p_{\sigma} \colon    \calh^{\calc}_n \circ W_* \circ I(\sigma)
    \to S_{\sigma} \bigl(\calh^{\calc}_n \circ W_* \circ I\bigr)
  \]
  has a section
  \[s_{\sigma} \colon S_{\sigma} \bigl(\calh^{\calc}_n \circ W_* \circ I\bigr)
    \to \calh^{\calc}_n \circ W_* \circ I(\sigma).
  \]
    Any collection of such sections and the canonical maps
    $\calh^{\calc}_n \circ W_* \circ I(\sigma) \to \colim_{\sigma \in \calp} \calh^{\calc}_n \circ W_* \circ I$
    induce an isomorphism
  \[
    \bigoplus_{\sigma \in \calp}  S_{\sigma} \bigl(\calh^{\calc}_n \circ W_* \circ I\bigr)
    \xrightarrow{\cong}
    \colim_{\sigma \in \calp} \calh_n^{W(\sigma)}(\EGF{W(\sigma)}{\calc}).
  \]
Moreover, there is an explicit section $s_\sigma$;

\item\label{the:main_theorem:Grothendieck}
  Suppose that each $\Lambda$-module $\calh_n^{W(\sigma)}(\EGF{W(\sigma)}{\calc})$ belongs to
  $\cale$. Then we get in $G_0(\cale)$
 \[
   [\calh_n^{W(\Sigma)}(\EGF{W(\Sigma)}{\calc})] = \sum_{\sigma \in \calp} n_{\sigma}
   \cdot [\calh_n^{W(\sigma)}(\EGF{W(\sigma)}{\calc})].
 \]
   
\end{enumerate}
\end{theorem}

We mention that it is both unusual and fortunate that in
assertion~\ref{the:main_theorem:iso_T} the source and target involve the same
degree. In general one would expect that in the source all degrees $m \le n$ occur. The
reason for this simplification is that for each simplex $\sigma$ the inclusion
$W(\sigma) \to W$ is split injective. This leads also to the explicit splitting in
assertion~\ref{the:main_theorem:splitting_the_colimit_iso_T}.

Note that the number $n_{\sigma}$ appearing in
assertion~\ref{the:main_theorem:Grothendieck} depends only on $\Sigma$.  It is given by
$1 -\chi(\Sigma)$ for $ \sigma = \emptyset$. It has the following geometric interpretation
if $\sigma$ is non-empty. Let $\Sigma'$ be the barycentric subdivision of $\Sigma$.  Then
$\ch_n(\Sigma,\sigma)$ can be interpreted as a collection of $n$-simplices in $\Sigma'$.
Each simplex in $\ch_n(\Sigma,\sigma)$ contains the vertex given by $\sigma$. The
collection of the faces of all these simplices of $\ch_n(\Sigma,\sigma)$ determines a
simplicial subcomplex $D_{\sigma}$ of $\Sigma'$ which can be contracted to the vertex
given by $\sigma$.  Its boundary $\partial D_{\sigma}$ consists of all those faces of
simplices of $\ch_n(\Sigma,\sigma)$ which do not contain $\sigma$. One easily checks
\[
n_{\sigma} = \chi(D_{\sigma}) - \chi(\partial D_{\sigma}) = 1 - \chi(\partial D_{\sigma}).
\]
If $\Sigma$ is the triangulation of a closed manifold of dimension $d$, then $\partial D_{\sigma}$
is homeomorphic to $S^{d - 1- \dim(\sigma)}$ and hence $n_{\sigma}  = (-1)^{d - \dim(\sigma)}$.


\subsection{Computations}
\label{subsec:Computations}

We will illustrate the potential of Theorem~\ref{the:main_theorem} by computing the group
homology of $G$, the algebraic $K$- and $L$-theory of the group ring of $G$, and the
topological $K$-theory of the group $C^*$-algebra of $G$ in
Sections~\ref{sec:right-angled_Artin_groups} and~\ref{sec:right-angled_Coxeter_groups} if
$G$ is a right-angled Artin group or a right-angled Coxeter group. These computations are based on the
Baum--Connes Conjecture and the Farrell--Jones Conjecture which we will briefly recall in
Section~\ref{sec:Brief_review_of_the_Baum-Connes_Conjecture_and_the_Farrell-Jones_Conjecture}
and which hold for these groups. The situation in the Farrell--Jones setting is more complicated
since we have to deal with the family $\calvcyc$ of virtually cyclic subgroups, whereas in
Theorem~\ref{the:main_theorem} the family $\calfin$ of finite subgroups is considered.
The passage from $\calfin$ to $\calvcyc$ is discussed in
Subsection~\ref{subsec:The_passage_from_calfin_to_calvcyc}. This is different in the
Baum--Connes setting since there the family $\calfin$ is used. In order to get full
functoriality we need to consider the maximal group $C^*$-algebra instead of the reduced
$C^*$-algebra which makes no difference for right-angled Artin groups and right-angled
Coxeter groups.

Computation in this context means not only that we identify the
corresponding $K$- and $L$-groups of $G$ as abelian groups but we give
explicit isomorphisms identifying them with $K$- and $L$-groups of the
ground ring.  For instance, we show for a right-angled Coxeter group
$W$ associated to the finite flag complex $\Sigma$ that there is for
every $n \in \IZ$ an isomorphism
  \[
    \bigoplus_{\sigma} K_n(f_{\sigma}) \colon \bigoplus_{\sigma} K_n(\IC) \xrightarrow{\cong}    K_n(C^*_r(W)),
    \\
   \]
   where $\sigma$ runs through the simplices of $\Sigma$ including the empty simplex and
   $f_{\sigma} \colon \IC \to C^*_r(W)$ is an explicit homomorphism of $C^*$-algebras
   depending on $\sigma$. If $\sigma$ is empty, it is given by the obvious inclusion
   $\IC \to C^*_r(W)$.  If $k = \dim(\sigma) \ge 0$, then $\sigma$ determines a subgroup
   $W(\sigma) = \prod_{i=1}^{k+1} \IZ/2$ of $W$ and $f_{\sigma}$  is the composite of
   the homomorphism $\IC \to \IC[\prod_{i=1}^{k+1} \IZ/2]$ sending $\lambda$ to
   $2^{-k-1} \cdot \lambda \cdot \prod_{i= 1}^{k+1} (1 -t_i)$ for $t_i$ the generator  of
   the $i$-th factor $\IZ/2$ and the homomorphism $\IC[W(\sigma)] \to C^*_r(W)$ coming
   from the inclusion $W(\sigma) \to W$, see Remark~\ref{rem:Explicite_isomorphism}.  This implies
  \[
     K_n(C^*_r(W))\cong
    \begin{cases}
      \IZ^r & \text{if}\; n \; \text{is even};
      \\
      \{0\} & \text{otherwise},
    \end{cases}
  \]
  where $r$ is the number of simplices of $\Sigma$ including the empty simplex. Moreover,
  we can write down an explicit basis $B = \{b_{\sigma} \mid \sigma \in \calp\}$ for the
  finitely generated free $\IZ$-module $K_0(C^*_r(W))$, namely, for $\sigma = \emptyset$
  we take the class of  the idempotent $1$ in $C^*_r(W)$ and for $\sigma \not= \emptyset$ we take the
  class of the idempotent in $C^*_r(W)$ given by the image of the idempotent
  $2^{-k-1} \cdot \prod_{i= 1}^{k+1} (1 -t_i) \in \IC[W(\sigma)]$ under the
  inclusion $\IC[W(\sigma)] \subset C^*_r(W)$. These computations for right-angled Coxeter groups were carried out in the second author's master's thesis~\cite{Li(2019)}.


\subsection{Acknowledgments}  
\label{subsec:Acknowledgements}

The paper was financially supported by the ERC Advanced Grant ``KL2MG-interactions'' (no. 662400) of the third author granted by the European Research Council. It was also funded by the Deutsche Forschungsgemeinschaft (DFG, German Research Foundation)
under Germany's Excellence Strategy - GZ 2047/1, Projekt-ID 390685813.

\tableofcontents


\typeout{-------------  Equivariant homology theories ----------------------}

\section{Equivariant homology theories and classifying spaces of families}
\label{sec:Equivariant_homology_theories}

In this section we recall the axioms of an equivariant homology theory and the notion of a classifying space of a family of subgroups. For an amalgamated product of groups, we deduce a Mayer--Vietoris type sequence for the values of an equivariant homology theory on classifying spaces.

Fix a discrete group $G$ and a commutative ring $\Lambda$ with unit.
A \emph{$G$-homology theory $\calh_*^G$ with values in $\Lambda$-modules} is a collection of covariant functors $\calh^G_n$ from the category of $G$-$CW$-pairs to the category of $\Lambda$-modules indexed by $n \in \IZ$ together with natural transformations
\[
	\partial_n^G(X,A)\colon  \calh_n^G(X,A) \to
	\calh_{n-1}^G(A)\coloneqq \calh_{n-1}^G(A,\emptyset)
\]
for $n \in \IZ$ such that the axioms concerning $G$-homotopy invariance, the long exact sequence of a pair, excision, and disjoint unions are satisfied, see~\cite[Section~1]{Lueck(2002b)}.

Let $\alpha\colon H \to G$ be a group homomorphism.  Given an $H$-space $X$, define the \emph{induction of $X$ with $\alpha$} to be the $G$-space
$  \alpha_* X%
   =  
        G \times_{\alpha} X%
       $
which is the quotient of $G \times X$ by the right $H$-action
$(g,x) \cdot h \coloneqq (g\alpha(h),h^{-1} x)$ for $h \in H$ and $(g,x) \in G \times X$.
The following definition is taken from~\cite[Section~1]{Lueck(2002b)} except that the induction structure in this paper is defined for every group homomorphism $\alpha$. 

\begin{definition}
\label{def:equivariant_homology_theory}
An \emph{equivariant homology theory $\calh^?_*$ with values in $\Lambda$-modules} assigns to each group $G$ a $G$-homology theory $\calh^G_*$ with values in $\Lambda$-modules together with the following so called \emph{induction structure}:%

Given a group homomorphism $\alpha\colon H \to G$ and an $H$-$CW$-pair $(X,A)$, there are for every $n \in \IZ$ natural homomorphisms
\begin{eqnarray*}
	&\ind_{\alpha} \colon  \calh_n^H(X,A) \to \calh_n^G(\alpha_* (X,A)) &
\end{eqnarray*}
satisfying:

\begin{itemize}

\item Compatibility with the boundary homomorphisms:\\[1mm]
$\partial_n^G \circ \ind_{\alpha} = \ind_{\alpha} \circ \partial_n^H$;

\item Functoriality:\\[1mm]
Let $\beta\colon G \to K$ be another group homomorphism.
Then we have for $n \in \IZ$
\[
\ind_{\beta \circ \alpha}  = \calh^K_n(f_1)\circ\ind_{\beta} \circ \ind_{\alpha} \colon
\calh^H_n(X,A) \to \calh_n^K((\beta\circ\alpha)_*(X,A)),
\]
where $f_1\colon  \beta_*(\alpha_*(X,A))
\xrightarrow{\cong} (\beta\circ\alpha)_*(X,A),
\quad (k,g,x) \mapsto (k\beta(g),x)$
is the natural $K$-homeo\-mor\-phism;

\item Compatibility with conjugation:\\[1mm]
For $n \in \IZ$, $g \in G$ and a  $G$-$CW$-pair $(X,A)$
the map \[\ind_{c(g)\colon  G \to G}\colon \calh^G_n(X,A)\to
\calh^G_n(c(g)_*(X,A))\] agrees with
$\calh_n^G(f_2)$ for the $G$-homeomorphism
$f_2\colon  (X,A) \to c(g)_*(X,A)$ which sends
$x$ to $(1,g^{-1}x)$ in $G\times_{c(g)} (X,A)$;

\item Bijectivity:\\[1mm]
If $\ker(\alpha)$ acts freely on $X\setminus A$, then 
$\ind_{\alpha}\colon  \calh_n^H(X,A) \to \calh_n^G(\alpha_*(X,A))$
is bijective for all $n \in \IZ$.
\end{itemize}
\end{definition}

We briefly fix some conventions concerning spectra.
If $X$ is a space, denote by $X_+$ the pointed space obtained from $X$ by adding a disjoint base point.
Let $\Spectra$ be the category of spectra in the following naive sense.
A \emph{spectrum}
$\mathbf{E} = \{(E(n),\sigma(n)) \mid n \in
\IZ\}$ is a sequence of pointed spaces
$\{E(n) \mid n \in \IZ\}$ together with pointed maps
called \emph{structure maps}
$\sigma(n) \colon  E(n) \wedge S^1 \longrightarrow E(n+1)$.
A \emph{map of spectra}
$\bff \colon  \bfE \to \bfE^{\prime}$ is a sequence of maps
$f(n) \colon  E(n) \to E^{\prime}(n)$
which are compatible with the structure maps $\sigma(n)$, i.e., we have
$f(n+1) \circ \sigma(n)  = 
\sigma^{\prime}(n) \circ \left(f(n) \wedge \id_{S^1}\right)$
for all $n \in \IZ$. Given a spectrum $\bfE$ and a pointed space $X$, we can define their
smash product $X \wedge \bfE$
by $(X \wedge \bfE)(n) \coloneqq X \wedge E(n)$ with the obvious structure
maps.

It is a classical result that a spectrum $\bfE$ defines a homology
theory by setting
\[
H_n(X,A;\bfE) = \pi_n\left((X_+ \cup_{A_+} \cone(A_+)) \wedge \bfE \right),
\]
where $\cone$ denotes the reduced cone.  We want to extend this to equivariant homology theories. 

Let $\Groupoids$  be the category of small connected groupoids with covariant functors as
morphisms.  Notice that a group can be considered as a groupoid with one object in the obvious way.

For the proof of the following result we refer to~\cite[Proposition~157 on  page~796]{Lueck-Reich(2005)}.

\begin{theorem}
\label{the:GROUPOID-spectra_and_equivariant_homology_theories}
Consider a covariant $\Groupoids$-spectrum
\[
\bfE\colon \Groupoids \to \Spectra.
\]
Suppose that $\bfE$ respects equivalences, i.e., it sends an
equivalence of groupoids to a weak equivalence of spectra.

Then $\bfE$ defines an equivariant homology theory $ H_*^?(-;\bfE)$
such that we have
\[
H^G_n(G/H; \bfE)  \cong  H^H_n(\pt; \bfE)  \cong  \pi_n(\bfE(H))
\]
for every group $G$, subgroup $H \subseteq G$ and $n \in \IZ$.
The construction is natural in $\bfE$.
\end{theorem}

\begin{example}[Borel homology]
\label{exa:Borel_homology}
Let $\bfE$ be a spectrum. Let $H_*(-;\bfE)$ be the (non-equivariant)
homology theory associated to $\bfE$.
Given a groupoid $\calg$, denote by $E\calg$ its classifying space.
If $\calg$ has only one object and the automorphism group of this object is $G$,
then $E\calg$ is a model for $EG$. We obtain two covariant functors
\begin{align*}
c_{\bfE} \colon \Groupoids \to \Spectra, &\quad \calg \mapsto \bfE;
\\
b_{\bfE} \colon \Groupoids \to \Spectra, &\quad \calg \mapsto E\calg_+ \wedge \bfE.
\end{align*} 
Thus we obtain two equivariant homology theories
$H^?_*(-;c_{\bfE})$ and $H^?_*(-;b_{\bfE})$ from 
Theorem~\ref{the:GROUPOID-spectra_and_equivariant_homology_theories}.
The second one is called the \emph{equivariant Borel homology associated to $H_*(-;\bfE)$}.
We get for any group $G$ and any $G$-$CW$-complex $X$ natural isomorphisms
\begin{eqnarray*}
H_n^G(X;c_{\bfE}) & \cong & H_n(G\backslash X;\bfE);
\\
H_n^G(X;b_{\bfE}) & \cong & H_n(EG \times_G X;\bfE).
\end{eqnarray*}
\end{example}

Let $\calc$ be a non-empty class of groups which is closed under isomorphisms, passage to
subgroups and passage to quotient groups.  Our main examples will be the class $\calfin$ of
finite groups and the class $\calvcyc$ of virtually cyclic groups. 

Given a group $G$,
denote by $\EGF{G}{\calc}$ the \emph{classifying space of $G$ with respect to the family of subgroups
$\calf_{\calc}(G) = \{ H \subseteq G \mid H \in \calc\}$}. 
It is defined to be a terminal object in the $G$-homotopy category of $G$-$CW$-complexes, whose isotropy groups belong to $\calf_{\calc}(G)$.
A model for $\EGF{G}{\calc}$ is a
$G$-$CW$-complex whose $H$-fixed point set is contractible if $H \in \calc$ and is empty
if $H \notin \calc$.  With this notation $\EGF{G}{\calfin}$ is the \emph{classifying space of
proper actions}, sometimes also denoted by $\eub{G}$. We sometimes denote
$\EGF{G}{\calvcyc}$ by $\edub{G}$. For more information about classifying spaces of
families we refer for instance to~\cite{Lueck(2005s)}.

Given a group homomorphism $f \colon G \to H $, we denote by
\begin{equation*}
 \EGF{f}{\calc} \colon f_*\EGF{G}{\calc}
 \to \EGF{H}{\calc}
\end{equation*}
the up to $H$-homotopy unique $H$-map coming from the universal property of
$\EGF{H}{\calc}$ and the fact that $f_*\EGF{G}{\calc}$ is an $H$-$CW$-complex whose
isotropy groups are of the shape $f(K)$ for $K \in \calc$ and hence all belong to $\calc$
again. Given an equivariant homology theory $\calh^?_*$ with values in $\Lambda$-modules,
it induces homomorphisms of $\IZ$-graded $\Lambda$-modules
\begin{equation*}
f_* \colon \calh^G_*(\EGF{G}{\calc}) \xrightarrow{\ind_f} \calh^H_*(f_*\EGF{G}{\calc})
\xrightarrow{\calh_*^H(\EGF{f}{\calc})}  \calh^H_*(\EGF{H}{\calc}).
\end{equation*}
One easily checks that thus we obtain a covariant functor 
\begin{equation}
  \calh^{\calc}_n \colon \Groups \to \Modcat{\Lambda}, \quad G \mapsto \calh^G_n(\EGF{G}{\calc}).
  \label{calh_upper_calc_n}
\end{equation}

\begin{lemma}\label{lem:amalgamated_products}
Let $G_0$, $G_1$ and $G_2$ be subgroups of $G$ satisfying $G_0 \subseteq G_1,G_2$.
Suppose that the inclusions $i_k \colon G_k \to G$ for $k =0,1,2$ induce an
isomorphism $G_1 \ast_{G_0} G_2 \xrightarrow{\cong} G$.
Let $j_k \colon G_0 \to G_k$ be the inclusion for $k = 1,2$.
Suppose that each element in $\calc$ is a finite group.

Then we obtain a long exact Mayer--Vietoris sequence
  \begin{multline*}
    \cdots \xrightarrow{\partial_{n+1}} \calh^{G_0}_n(\EGF{G_0}{\calc}) \xrightarrow{(j_1)_n \times - (j_2)_n}
    \calh^{G_1}_n(\EGF{G_1}{\calc})  \oplus \calh^{G_2}_n(\EGF{G_2}{\calc})
    \\
    \xrightarrow{(i_1)_n \oplus (i_2)_n}
    \calh^{G}_n(\EGF{G}{\calc})  \xrightarrow{\partial_n} \calh^{G_0}_{n-1}(\EGF{G_0}{\calc})
    \xrightarrow{(j_1)_{n-1} \times -(j_2)_{n-1}} \cdots.
  \end{multline*}
  \end{lemma}
  \begin{proof}
  There is a $1$-dimensional $G$-$CW$-complex $T$ whose underlying space is a tree
such that the $1$-skeleton is obtained from the $0$-skeleton by the $G$-pushout
\[
\xymatrix{G/G_0 \times S^0  \ar[r]^q \ar[d]
  &
  G/G_1 \amalg G/G_2 \ar[d]
\\
G/G_0 \times D^1 \ar[r]
&
T
}
\]
where $q$ is the disjoint union of the canonical projections $G/G_0 \to G/G_1$ and
$G/G_0 \to G/G_2$, see~\cite[Theorem~7 in~\S 4.1 on page~32]{Serre(1980)}.  If we take the
cartesian product with $\EGF{G}{\calc}$ we obtain another cellular $G$-pushout.  Its
associated Mayer--Vietoris sequence yields the long exact sequence
\begin{multline*}
    \cdots\to  \calh^{G}_n(G/G_0 \times \EGF{G}{\calc}) \to
    \calh^{G}_n(G/G_1 \times \EGF{G}{\calc})  \oplus \calh^{G}_n(G/G_2 \times \EGF{G}{\calc})
    \\
    \to
    \calh^{G}_n(T \times \EGF{G}{\calc})  \to  \calh^{G}_{n-1}(G/G_0 \times \EGF{G}{\calc}) \to \cdots.
  \end{multline*}
  There is a $G$-homeomorphism
  $(j_k)_*j_k^* \EGF{G}{\calc} \xrightarrow{\cong} G/G_k \times \EGF{G}{\calc}$,
  where $j_k^* \EGF{G}{\calc}$ is the restriction of $\EGF{G}{\calc}$ to $G_k$ by
  $j_k$. Obviously $ j_k^* \EGF{G}{\calc}$ is a model for $\EGF{G_k}{\calc}$. Using the
  induction structure of the equivariant homology theory $\calh_*^{?}$ we obtain
  identifications for $k = 0,1,2$
  \[
  \calh^{G_k}_*(\EGF{G_k}{\calc}) \xrightarrow{\cong} \calh^G_*(G/G_k \times \EGF{G}{\calc}).
  \]
  The $K$-fixed point set $T^K$ is a non-empty subtree and hence contractible for every
  finite subgroup $K \subseteq G$,
  see~\cite[Theorem~15 in~6.1 on page~58 and~6.3.1 on  page~60]{Serre(1980)}.
  Hence the projection   $T \times \EGF{G}{\calc} \to \EGF{G}{\calc}$ is a $G$-homotopy equivalence
  since every element in $\calc$ is finite by assumption. Hence we get an identification
  \[
  \calh^{G}_*(\EGF{G}{\calc}) \xrightarrow{\cong} \calh^G_*(T \times \EGF{G}{\calc}).
  \]
  Now we obtain the desired long exact sequence from the last long exact sequence and the
  identifications above.
\end{proof}


\typeout{------------- Graph product of groups ----------------------}

\section{Graph products of groups}
\label{sec:Graph_product_of_groups}

In this section we give the definition of a graph product of groups. We show that the value of an equivariant homology theory on the classifying space of a graph product is the colimit over a certain system of subgroups.

Let $X$  be a finite simplicial graph on the vertex set $V$  and suppose that we are
given a collection of groups $\calw \coloneqq \{W_v \mid v \in V\}$.
Then the \emph{graph product} $W(X,\calw)$ is defined as the quotient of
 the free product $\ast_{v \in V} W_v$ of the collection of groups $\calw$  by introducing the relations
\[
  \{[g,g'] = 1\mid v, v' \in V, \;\text{there is an edge joining} \;v \; \text{and} \;v', g \in W_v, g' \in W_{v'}
  \}.
  \]
In other words, elements of subgroups $W_v$ and $W_{v'}$ commute if there is an
edge joining $v$ and $v'$. This notion is due to Green~\cite{Geen(1990)}.

Let $\Sigma$ be the flag complex associated to  $X$. Denote by $\cals = \cals(\Sigma)$ the poset of flag subcomplexes of $\Sigma$, ordered by inclusion, where we also allow the empty
subcomplex.  Then we can assign to $L \in \cals$ the graph product group
$W(X \cap L,\calw|_{V\cap L})$, where $X \cap L$ agrees with  the $1$-skeleton of $L$ and
$\calw|_{V\cap L} = \{W_v \mid v \in V \cap L\}$ is the restriction of $\calw$ to the
vertices in $L$.  Consider $L_0, L_1 \in \cals$ with $L_0 \le L_1$.  Then we
obtain group homomorphisms
\begin{eqnarray*}
  W_*(L_0 \le L_1) \colon W(L_0,\calw|_{V\cap L_0}) & \to & W(L_1,\calw|_{V\cap L_1});
  \\
  W^*(L_0 \le L_1) \colon W(L_1,\calw|_{V\cap L_1}) & \to & W(L_0,\calw|_{V\cap L_0})
\end{eqnarray*}
as follows. The morphism $W_*(L_0 \le L_1)$ is induced by the obvious inclusion
$\ast_{v \in V \cap L_0} W_v\to \ast_{v \in V \cap L_1 } W_v$, whereas the second one
is induced by the projection
$\ast_{v \in V \cap L_1} W_v \to \ast_{v \in V \cap L_0} W_v$ which is given on $W_v$
for $v \in V \cap L_1$ by the inclusion
$W_v \to \ast_{v \in V \cap L_0} W_v$ if $v \in L_0$, and by the trivial homomorphism
if $v \notin L_0$.  One easily checks that thus we obtain a covariant functor
\begin{equation*}
W_* \colon \cals \to \Groups
\end{equation*}
and a contravariant functor
\begin{equation*}
W^* \colon \cals \to \Groups.
\end{equation*}
By construction $W_*$ and $W^*$ agree on objects and we write
\[
  W(L) \coloneqq W_*(L) = W^*(L) \coloneqq W(X \cap L,\calw|_{V \cap L})
\]
for an object $L\in \cals$.

The elementary proof of the following lemma is left to the reader.

\begin{lemma}\label{lem:basic_properties_of_W_ast_and_W_upper_ast}\
\begin{enumerate}

\item\label{lem:basic_properties_of_W_ast_and_W_upper_ast:amalgamation}

  Let $L_0, L_1,L_2, L \in \cals$ be elements such that $L = L_1 \cup L_2$ and $L_0 = L_1 \cap L_2$.
  Then we obtain a group isomorphism
  \[
   W_*(L_1 \le L) \ast_{W_*(L_0 \le L)} W_*(L_2 \le L) \colon W(L_1) \ast_{W(L_0)} W(L_2) \xrightarrow{\cong} W(L);
  \]

\item\label{lem:basic_properties_of_W_ast_and_W_upper_ast:Mackey}
  Let
  $L_1,L_2,L \in \cals$ be elements satisfying $L_1 \le L$ and $L_2 \le L$.
  Then we get an equality of group homomorphisms $W(L_1) \to W(L_2)$
  \[
  W^*(L_2 \le L) \circ W_*(L_1 \le L) = W_*((L_1 \cap L_2) \le L_2) \circ W^*((L_1 \cap L_2) \le L_1).
  \]
\end{enumerate}
\end{lemma}

\begin{remark}\label{rem:retratcions}
Notice that in particular we get from
Lemma~\ref{lem:basic_properties_of_W_ast_and_W_upper_ast}~%
\ref{lem:basic_properties_of_W_ast_and_W_upper_ast:Mackey} that for any two elements $L_0$
and $L_1$ in $\cals$ with $L_0 \le L_1$ the composite
$W^*(L_0 \le L_1) \circ W_*(L_0 \le L_1)$ is the identity on $W(L_0)$ and hence
$W_*(L_0 \le L_1)$ is split injective and $W^*(L_0 \le L_1)$ is split surjective.
\end{remark}

Let $\calh^?_*$ be an equivariant homology theory with values in $\Lambda$-modules.
Let $\calc$ be a non-empty class of groups which is closed under isomorphisms, passage to subgroups
and passage to quotient groups. We have defined a covariant functor
$\calh^{\calc}_*$ and studied its main properties in
Section~\ref{sec:Equivariant_homology_theories}.

We want to study the covariant $\Lambda\cals$-module
\[
  \calh_n^{\calc} \circ W_* \colon \cals \to \Modcat{\Lambda}, \quad
L \mapsto \calh_n^{W(L)}(\EGF{W(L)}{\calc})
\]
and are in particular interested in its value at $\Sigma$ itself.

Viewing a simplex as a flag subcomplex yields for every $L \in \cals$ a map of posets
\[
I_{L} \colon \calp(L) \to \cals(L).
\]
For two elements $L$ and $L'$ in $\cals$ with $L \le L'$ let 
\begin{eqnarray*}
  J^{\calp}(L \le L')\colon \calp(L)  &\to & \calp(L');
  \\
  J^{\cals}(L \le L') \colon \cals(L)  &\to & \cals(L')
\end{eqnarray*}                                    
be the maps of posets induced by the inclusion $L \subseteq L'$.
Define the  $\Lambda$-module
\[C_n(L) \coloneqq \colim_{\calp(L)} \calh_n^{\calc} \circ W_* \circ  J^{\cals}(L \le \Sigma) \circ I_L
\]
to be the colimit of the covariant functor
$\calh_n^{\calc} \circ W_* \circ  J^{\cals}(L \le \Sigma) \circ I_L \colon \calp(L) \to \Modcat{\Lambda}$.
Given elements $L$ and $L'$ in $\cals$ with $L \le L'$,
we obtain a map of $\Lambda$-modules
\[C_n(L \le L') \colon C_n(L) \to C_n(L')
\]
from $J^{\calp}(L \le L')$ because  of $J^{\cals}(L' \le \Sigma) \circ I_{L'} \circ J^{\calp}(L \le L')= J^{\cals}(L \le \Sigma) \circ I_{L}$.
One easily checks that thus we obtain a covariant $\Lambda\cals$-module
\[
  C_n \colon \cals \to \Modcat{\Lambda}, \quad L \mapsto C_n(L).
\]
For every object $L\in\cals$ there is an obvious $\Lambda$-homomorphism
\[
  T_n(L) \colon C_n(L) = \colim_{\calp(L)} \calh_n^{\calc} \circ W_* \circ  J^{\cals}(L \le \Sigma) \circ I_L \to
  \calh_n^{\calc} \circ W_*(L)
\]
coming from the various $\Lambda$-maps
$\calh_n^{\calc} \circ W_*(\sigma) \to \calh_n^{\calc} \circ W_*(L)$
induced by the inclusions $I_L(\sigma) \subseteq L$ for $\sigma$
running through the simplices of $L$.  One easily checks that the collection of the $\Lambda$-homomorphisms $T_n(L)$ fits together to a map
of covariant $\Lambda\cals$-modules
\begin{equation}
  T_n \colon C_n \to \calh_n^{\calc} \circ W_*.
  \label{T:C_n_to_calh_n_upper_calc_circ_W_ast}
\end{equation}

\begin{theorem}
\label{Computing_calh_upper_W(Sigma)(EGF(W(Sigma))(calc)}
Suppose that each element in $\calc$ is a finite group.

Then the map of $\Lambda\cals$-modules $T_n$
of~\eqref{T:C_n_to_calh_n_upper_calc_circ_W_ast} is an isomorphism.  In
particular its evaluation at $\Sigma$ yields a $\Lambda$-isomorphism
  \[T_n(\Sigma) \colon \colim_{\sigma \in \calp} \calh_n^{W(\sigma)}(\EGF{W(\sigma)}{\calc})
  \xrightarrow{\cong} \calh^{W(\Sigma)}_n(\EGF{W(\Sigma)}{\calc}).
  \]
\end{theorem}
\begin{proof}
	Notice that $T_n(L)$ is obviously an isomorphism if $L$ lies in the image of $I \colon \calp \to \cals$, since then $L$ is a terminal object in $\calp(L)$ and hence under the obvious identification $C_n(L) = \colim_{\calp(L)} \calh_n^{\calc} \circ W_* \circ J^{\cals}(L \le \Sigma) \circ I_L \cong \calh_n^{\calc} \circ W_*(L)$ the $\Lambda$-homomorphism $T_n(L)$ becomes the identity.
  
	We show for any $L \in \cals$ that $T_n(L)$ is an isomorphism by induction over the number of vertices of $L$. 
	If $L$ is the empty subcomplex, then $L$ is in the image of $I \colon \calp \to \cals$ and the claim has already been proved.
	The induction step is done as follows. We only have to consider the case, where $L$ is not in the image of $I \colon \calp \to \cals$. Since $L$ is a flag complex, there must be two vertices $v_1$ and $v_2$ in $L$ which are not connected by an edge. Let $L_1$ be the flag subcomplex of $L$ spanned by $v_1$ and all vertices in $L$ which are connected to $v_1$ by an edge. In particular $v_2$ is not a vertex of $L_1$. Let $L_0$ be the flag subcomplex of $L$ which is spanned by all vertices $v$ for which there exists an edge whose terminal points are $v$ and $v_1$. Notice that $v_1$ does not belong to $L_0$ and $L_1$ is the cone over $L_0$ with cone point $v_1$. Let $L_2$ be the flag subcomplex of $L$ spanned by all vertices except $v_1$. Then $L = L_1 \cup L_2$ and $L_0 = L_1 \cap L_2$ and the number of vertices of $L_0$, $L_1$ and $L_2$ is smaller than the number of vertices of $L$. The induction hypothesis applies to $L_k$ and hence $T_n(L_k)$ is an isomorphism for $k = 0,1,2$. 
	Since $\calp(L) = \calp(L_1) \cup \calp(L_2)$ and $\calp(L_0) = \calp(L_1) \cap \calp(L_2)$, the sequence induced by the inclusions $\calp(L_0) \subseteq \calp(L_k)$ and $\calp(L_k) \subseteq \calp(L)$ for $k = 1,2$
  	\[
  		C_n(L_0) \to  C_n(L_1) \oplus C_n(L_2) \to C_n(L) \to 0
  	\]
  	is exact. Since the group homomorphism $W_*(L_0) \to W_*(L_1)$ is split injective, a retraction is given by $W^*(L_1) \to W^*(L_0)$, the $\Lambda$-homomorphism $\calh^{\calc}_n\circ W_*(L_0 \le L_1) \colon \calh^{\calc}_n\circ W_*(L_0) \to \calh^{\calc}_n\circ W_*(L_1)$ is split injective by functoriality. We get from Lemma~\ref{lem:amalgamated_products} and Lemma~\ref{lem:basic_properties_of_W_ast_and_W_upper_ast}~%
\ref{lem:basic_properties_of_W_ast_and_W_upper_ast:amalgamation}
  	an exact sequence
  \[
    \calh_n^{\calc} \circ W_*(L_0) \to   \calh_n^{\calc} \circ W_*(L_1) \oplus  \calh_n^{\calc} \circ W_*(L_2)
    \to  \calh_n^{\calc} \circ W_*(L) \to 0.
\]
One easily checks that we obtain a commutative diagram with exact rows
\[
  \xymatrix{C_n(L_0) \ar[r] \ar[d]^{T_n(L_0)}_{\cong}
 &
 C_n(L_1) \oplus C_n(L_2) \ar[r]  \ar[d]^{T_n(L_1) \oplus T_n(L_0)}_{\cong} &
 C_n(L) \ar[r] \ar[d]^{T_n(L)}
 &
 0
 \\
\calh_n^{\calc} \circ W_*(L_0) \ar[r] 
 &
 \calh_n^{\calc} \circ W_*(L_1) \oplus\calh_n^{\calc} \circ W_*(L_2) \ar[r]   &
 \calh_n^{\calc} \circ W_*(L) \ar[r] 
 &
 0
}
\]
Now the induction step follows from the Five-Lemma.
\end{proof}

We have proven part~\ref{the:main_theorem:iso_T} of the main Theorem~\ref{the:main_theorem}.


\typeout{------------- Mackey modules ----------------------}

\section{Mackey modules}
\label{sec:Mackey_modules}

In this section we prove the remaining parts of the main Theorem~\ref{the:main_theorem}. More generally, we show that the colimit of any Mackey module splits as a direct sum over its index category.

Let $\Sigma$ be a finite simplicial complex. Denote by $\calp = \calp(\Sigma)$ the  poset of its simplices ordered by inclusion, where we allow the empty simplex as well. The dimension of the empty simplex is defined to be $-1$.
Notice that for two elements $\sigma$ and $\tau$ in $\calp$ the intersection $\sigma \cap \tau$ is again an element in $\calp$ which is uniquely determined by the property that it is maximal among all those elements $\mu$ in $\calp$ satisfying both $\mu\le \sigma$ and $\mu \le \tau$.

Let $\Lambda$ be a commutative ring with unit. A \emph{Mackey $\Lambda\calp$-module } $M = (M_*,M^*)$ is a bifunctor
$\calp \to \Modcat{\Lambda}$, i.e., a covariant functor $M_*$ and a contravariant
functor $M^*$ from $\calp$ to $\Modcat{\Lambda}$ such that $M_*$ and $M^*$ agree
on objects and for objects $\sigma_1$, $\sigma_2$, $\tau$ of $\calp$
satisfying $\sigma_k \le \tau$ for $k = 1,2$, we get
\begin{eqnarray}
  M^*(\sigma_2 \le \tau) \circ M_*(\sigma_1 \le \tau)
  & = &
 M_*(\sigma_1 \cap \sigma_2  \le \sigma_2) \circ M^*(\sigma_1 \cap \sigma_2  \le \sigma_1).
\label{double_coset_formula}
\end{eqnarray}
The name Mackey $\Lambda\calp$-module comes from the analogy to the classical notion of a
Mackey functor, where~\eqref{double_coset_formula} replaces the double coset formula,
see~\cite[Section~6.1]{Dieck(1979)}.
for an object $\tau$ of $\calp$.

\begin{example}[Mackey modules coming from graph products]
\label{exa:Mackey_modules_coming_from_graph_products}
	Our main example comes from Section~\ref{sec:Graph_product_of_groups}.
	Let $X$  be a finite simplicial graph on the vertex set $V$  and suppose that we are given a collection of groups $\calw \coloneqq \{W_v \mid v \in V\}$. Let $F \colon \Groups \to \Modcat{\Lambda}$ be a covariant functor, e.g.\ the functor $\calh^\calc_*$ defined in~\eqref{calh_upper_calc_n}. Define $M_* = F \circ W_* \circ I$ and $M^* = F \circ W^* \circ I$. Then the pair $(M_*,M^*)$ defines a Mackey $\Lambda\calp$-module by Lemma~\ref{lem:basic_properties_of_W_ast_and_W_upper_ast}~%
	\ref{lem:basic_properties_of_W_ast_and_W_upper_ast:Mackey}.
\end{example}

Fix elements $\tau$ in $\calp$ and $d \in \{-2,-1,0,1,2, \ldots\}$.
Consider a covariant $\Lambda\calp$-module $N$, i.e., a covariant functor
$N \colon \calp \to \Modcat{\Lambda}$.  Define the
$\Lambda$-submodules $L^d_{\tau}N$ and $L_{\tau} N$ of $N(\tau)$ to be the
images of the maps
\[
  \bigoplus_{\substack{\sigma \in \calp, \sigma < \tau\\\dim(\sigma) \le d}} N(\sigma \le \tau)  \colon
\bigoplus_{\substack{\sigma \in \calp, \sigma <\tau \\\dim(\sigma) \le d}} N(\sigma) \to N(\tau)
\]
and
\[
\bigoplus_{\sigma \in \calp, \sigma <\tau} N(\sigma\le \tau) \colon \bigoplus_{\sigma \in \calp, \sigma < \tau} N(\sigma)
\to N(\tau),
\]
respectively. Define $\Lambda$-quotient modules of $N(\tau)$ by
\begin{eqnarray*}
S^d_{\tau} N & = & N(\tau)/L^d_{\tau}(N);
  \\
S_{\tau} N & = & N(\tau)/L_{\tau}(N).
\end{eqnarray*}
Then we obtain a sequence of inclusions of $\Lambda$-modules
\[
\{0\} = L^{-2}_{\tau}N\subseteq L^{-1}_{\tau}N\subseteq  L^{0}_{\tau}N \subseteq L^{1}_{\tau}N
\subseteq \cdots \subseteq L^{\dim(\tau)-1}_{\tau}N = L_{\tau}N,
\]
and a sequence of epimorphisms of $\Lambda$-modules
\[
N(\tau) = S^{-2}_{\tau} N \to  S^{-1}_{\tau} N \to  S^{0}_{\tau} N  \to S^{1}_{\tau} N \to \cdots
\to S^{\dim(\tau) -1}_{\tau} N  = S_{\tau} N.
\]
Note that $L^{-1}_{\tau}N = \im\bigl(N(\emptyset) \to N(\tau)\bigr)$ and $S^{-1}_{\tau}N = \cok\bigl(N(\emptyset) \to N(\tau)\bigr)$
for $\tau \not= \emptyset$ and $L^{-1}_{\emptyset}N = \{0\}$ and $S^{-1}_{\emptyset}N = N(\emptyset)$.

Consider a Mackey $\Lambda\calp$-module $M = (M_*,M^*)$. Define a $\Lambda$-homomorphism
\[s^d_{\tau} \colon M_*(\tau) \to M_*(\tau)
\]
by
\[s^d_{\tau} \coloneqq \id _{M_*(\tau)} -\sum_{\substack{\mu \in \calp, \mu< \tau\\ \dim(\mu) = d}}
M_*(\mu \le \tau) \circ M^*(\mu \le \tau).
\]
\begin{lemma}
We have for $d \in \{-1,0,1,2, \ldots\}$
\[
  s^d_{\tau}\bigl(L^d_{\tau} M_*\bigr) \subseteq L^{d-1}_{\tau} M_*.
\]
\label{main_property_of_s_upper_d}
\end{lemma}
\begin{proof}
We compute for $\sigma \in \calp$ satisfying $\sigma < \tau$ and $\dim(\sigma) \le d$
\begin{eqnarray*}
  \lefteqn{s^d_{\tau}\circ M_*(\sigma \le \tau)}
  & &
   \\
  & = &
  M_*(\sigma \le \tau) - \sum_{\substack{\mu \in \calp, \mu < \tau\\ \dim(\mu) = d}}
  M_*(\mu \le \tau) \circ M^*(\mu \le \tau) \circ  M_*(\sigma \le \tau)
  \\
  & \stackrel{\eqref{double_coset_formula}}{=} &
  M_*(\sigma \le \tau) - \sum_{\substack{\mu \in \calp, \mu < \tau\\ \dim(\mu) = d}}
  M_*(\mu \le \tau) \circ   M_*(\mu \cap \sigma \le \mu) \circ M^*(\mu \cap \sigma \le \sigma)
  \\
  & = &
  M_*(\sigma \le \tau) - \sum_{\substack{\mu \in \calp, \mu < \tau\\ \dim(\mu) = d}}
  M_*(\mu \cap \sigma \le \tau ) \circ M^*(\mu \cap \sigma \le \sigma)
  \\
  & = &
  M_*(\sigma \le \tau) - \sum_{\substack{\mu \in \calp, \mu < \tau\\ \dim(\mu) = d, \dim(\mu \cap \sigma) =  d}}
  M_*(\mu \cap \sigma \le \tau ) \circ M^*(\mu \cap \sigma \le \sigma)
  \\
  & & \quad
  - \sum_{\substack{\mu \in \calp, \mu < \tau\\ \dim(\mu) = d, \dim(\mu \cap \sigma) \le d-1}}
 M_*(\mu \cap \sigma \le \tau ) \circ M^*(\mu \cap \sigma \le \sigma).
\end{eqnarray*}
Suppose that $\dim(\mu \cap \sigma) =  d$. Since $\dim(\mu) = d$ and $\dim(\sigma) \le  d$,
we conclude $\mu = \mu \cap \sigma =  \sigma$ and hence because of $\sigma < \tau$
\[
  \{\mu \in \calp \mid \mu < \tau, \dim(\mu) = d, \dim(\mu \cap \sigma) =  d\} = \{\sigma\}.
\]
This implies
\[M_*(\sigma \le \tau) - \sum_{\substack{\mu \in \calp, \mu < \tau\\ \dim(\mu) = d, \dim(\mu \cap \sigma) =  d}}
  M_*(\mu \cap \sigma \le \tau ) \circ M^*(\mu \cap \sigma \le \sigma) = 0,
\]
and hence
\[s^d_{\tau}\circ M_*(\sigma \le \tau)
  =
- \sum_{\substack{\mu \in \calp, \mu < \tau\\ \dim(\mu) = d, \dim(\mu \cap \sigma) \le d-1}}
M_*(\mu \cap \sigma \le \tau ) \circ M^*(\mu \cap \sigma \le \sigma).
\]

We conclude for every $\sigma \in \calp$ satisfying $\sigma < \tau$ and $\dim(\sigma) \le d$
\[
 \im\bigl(s^d_{\tau} \circ M_*(\sigma \le \tau)\bigr) \subseteq L^{d-1}_{\tau} M_*.
\]
Now the assertion follows.
\end{proof}

Define a map
\[s_{\tau} \colon M_*(\tau) \to  M_*(\tau)\]
to be the composite $s^{-1}_{\tau} \circ  s^0_{\tau}  \circ \cdots \circ s^{\dim(\tau)-1}_{\tau}$.
Then we conclude by induction from Lemma~\ref{main_property_of_s_upper_d}
that $s_{\tau}$ restricted to $L_{\tau} M_*$ is trivial and hence induces a $\Lambda$-homomorphism
\begin{equation*}
\overline{s}_{\tau} \colon S_{\tau} M_* \to M_*(\tau).
\end{equation*}
    
\begin{lemma}\label{lem:overline(s)_tau_M_is_a_section}
Let  $M = (M_*,M^*)$ be a Mackey $\Lambda\calp$-module.
Consider an element $\tau \in \calp$. Let $p_{\tau} \colon M_*(\tau) \to S_{\tau}M_*$
be the projection. Then
\[
  p_{\tau}\circ \overline{s}_{\tau} = \id_{S_{\tau} M_*}.
\]
\end{lemma}
\begin{proof} Obviously each map $s^d_{\tau}$ satisfies $p_{\tau} \circ s^d_{\tau} = p_{\tau}$.
\end{proof}

For every element $\tau \in \calp$ the restriction 
\[
	\RES_\tau\colon \Modcat{\Lambda\calp}\to \Modcat{\Lambda}
\]
has a left adjoint $E_\tau$ by~\cite[Lemma~9.31 on page~171]{Lueck(1989)}.
Explicitly, the functor $E_\tau \colon \Modcat{\Lambda} \to \Modcat{\Lambda\calp}$ is given by $E_\tau(N)=N \otimes_\Lambda \Lambda\mor_\calp(\tau, {?})$, where ${?}$ runs through the objects of $\calp$ and $\Lambda\mor_{\calp}(\tau, {?})$ is the free $\Lambda$-module with the set $\mor_{\calp}(\tau, {?})$ as basis.
Equivalently, $N \otimes_{\Lambda} \Lambda\mor_{\calp}(\tau, {?})$ assigns to an object ${?}$ the $\Lambda$-module $N$ if $\tau \le {?}$, and $\{0\}$ otherwise. Functoriality in ${?}$ is given by the identity on $\{0\}$ or $N$, or by the inclusion $\{0\} \to N$.

We get for every element $\tau\in\calp$ and every map of covariant $\Lambda$-modules $u \colon S_\tau M_* \to M_*(\tau)$ a map of covariant $\Lambda\calp$-modules
\[
	\ad(u) \colon E_\tau(S_\tau M_*) \to M_*
\]
by the adjoint of $u$ under the adjunction $(E_\tau, \RES_\tau)$.
For $?\in \calp$ the map $\ad(u)({?})$ 
is given by the composite $S_{\tau} M_* \xrightarrow{u} M_*(\tau) \xrightarrow{M_*(\tau \le {?})} M_*(?)$ if
 $\tau \le {?}$ and by the inclusion $\{0\} \to M_*(?)$ otherwise.

\begin{lemma}\label{lem:widehat(s)_bijective}
	Consider any collection of homomorphisms $\widehat{s}_{\tau} \colon S_{\tau} M_* \to M_*(\tau)$ satisfying $p_{\tau} \circ \widehat{s}_{\tau} = \id_{S_{\tau} M_*}$, where $\tau$ runs through the elements in $\calp$. 
   
   Then the  homomorphism of covariant $\Lambda \calp$-modules
   \[\ad(\widehat{s}) \coloneqq \bigoplus_{\tau \in \calp} \ad(\widehat{s}_{\tau}) \colon
     \bigoplus_{\tau \in \calp} E_{\tau}(S_{\tau}M_*) \to M_*
  \]
  is an isomorphism.
\end{lemma}
\begin{proof}
  We start with injectivity. Suppose that $\ad(\widehat{s})$ is not injective.
  Then there exists an element $\sigma \in \calp$ and a non-trivial
  element $a = (a_{\tau})_{\tau \in \calp} \in \bigoplus_{\tau \in \calp} E_{\tau}(S_{\tau}M_*)(\sigma)$
  such that $\ad(\widehat{s})(\sigma)(a) = 0$. Choose $\tau_0 \in \calp$ with $a_{\tau_0} \not= 0$ such that for all $\tau\in \calp$ with $a_{\tau} \not=0$ we have $\dim(\tau) \le \dim(\tau_0)$.
  If $\tau\in \calp$ satisfies $a_{\tau} \not= 0$ and hence
    $E_{\tau}(S_{\tau}M_*)(\sigma) \not= 0$, we use the explicit description of $E_\tau(S_\tau M_*)$ to conclude $\tau \le \sigma$.
  The composite
  \begin{equation}
  \bigoplus_{\tau \in \calp} E_{\tau}(S_{\tau}M_*)(\sigma) \xrightarrow{\ad(\widehat{s})(\sigma)} M_*(\sigma)
  \xrightarrow{M^*(\tau_0 \le \sigma)} M_*(\tau_0) \xrightarrow{p_{\tau_0}} S_{\tau_0} M_*
  \label{lem:widehat(s)_bijective:(1)}
  \end{equation}
  sends $a$ to $0$ since $\ad(\widehat{s})(\sigma)(a) = 0$.

  Consider an element $\tau  \in \calp$ with $a_{\tau} \not= 0$.
  Then $\dim(\tau) \le \dim(\tau_0)$ and hence we get $\tau \cap \tau_0 < \tau_0$ if
  $\tau\not= \tau_0$ and $\tau \cap \tau_0 = \tau_0$ if $\tau = \tau_0$.  The
  composite
  \[
   M_*(\tau) \xrightarrow{M_*(\tau \le \sigma)} M_*(\sigma)
    \xrightarrow{M^*(\tau_0 \le \sigma)} M_*(\tau_0) \xrightarrow{p_{\tau_0}} S_{\tau_0} M_*
  \]
  agrees because of~\eqref{double_coset_formula}  with the composite
  \[
    M_*(\tau) \xrightarrow{M^*(\tau \cap \tau_0\le \tau)} M_*(\tau\cap \tau_0)
    \xrightarrow{M_*(\tau\cap \tau_0 \le \tau_0)} M_*(\tau_0) \xrightarrow{p_{\tau_0}} S_{\tau_0} M_*.
  \]
  Hence it is zero if $\tau \not= \tau_0$ and it is $p_{\tau_0}$ if $\tau = \tau_0$.
  This implies that the restriction of the composite~\eqref{lem:widehat(s)_bijective:(1)}
  to the summand associated to $\tau$ is trivial if $\tau \not= \tau_0$ and is the
  identity under the obvious
  identification $E_{\tau_0}(S_{\tau_0}M_*)(\sigma) = S_{\tau_0} M_*$ if $\tau = \tau_0$.
  We conclude  that the
  composite~\eqref{lem:widehat(s)_bijective:(1)} sends $a$ to $a_{\tau_0}$
  under the obvious identification $E_{\tau_0}(S_{\tau_0}M_*)(\sigma) = S_{\tau_0} M_*$.
  Since this implies $a_{\tau_0} = 0$, we get a contradiction. This finishes the proof that
$\ad(\widehat{s})$ is injective.

Next we show by induction for $d = -1,0,1,2, \ldots$ that $\ad(\widehat{s})(\sigma)$ is
surjective for all $\sigma \in \calp$ with $\dim(\sigma) \le d$. The induction beginning
is obvious since $\emptyset$ is the unique  initial object,  hence
$S_{\emptyset}M_*= M_*(\emptyset)$ holds and therefore $\ad(\widehat{s})_{\emptyset}(\emptyset)$ is  bijective.
The induction step from $(d-1)$ to
$d \ge 0$ is done as follows.  The composite
\[
  E_{\sigma}(S_{\sigma}M_*)(\sigma) \xrightarrow{\ad(\widehat{s})_{\sigma}} M_*(\sigma) \to S_{\sigma}M_*
\]
is surjective. Hence it suffices to show that $L_{\sigma}M_*$ is contained in the
restriction of $\ad(\widehat{s})(\sigma)$ to
$\bigoplus_{\tau \in \calp, \tau \not= \sigma} E_{\tau}(S_{\tau} M_*)(\sigma)$.  It
suffices to show that for every $\sigma'$ with $\sigma' < \sigma$ the image of
$M_*(\sigma') \to M_*(\sigma)$ is contained in the restriction of $\ad(\widehat{s})(\sigma)$
to $\bigoplus_{\tau \in \calp, \tau \not= \sigma} E_{\tau}(S_{\tau} M_*)(\sigma)$.  By
induction hypothesis $\ad(\widehat{s})(\sigma')$ is surjective.  Now
Lemma~\ref{lem:widehat(s)_bijective} follows from naturality of $\ad(\widehat{s})$ and the
fact that $E_{\sigma}(S_{\sigma} M_*)(\sigma')$ vanishes.
\end{proof}

For a covariant $\Lambda\calp$-module $N$, denote by $H_n(\calp;N)$ its homology.  This is
$H_n(P_* \otimes_{\Lambda\calp} N)$ for the $\Lambda$-homology of the $\Lambda$-chain complex
$P_* \otimes_{\Lambda\calp} N$ for any projective $\Lambda\calp$-resolution $P_*$ of the constant
$\Lambda\calp$-module $\underline{\Lambda}$ with value $\Lambda$. In the notation
of~\cite[Chapter~17]{Lueck(1989)} this is
$\operatorname{Tor}_n^{\Lambda\calp}(\underline{\Lambda},N)$.

\begin{theorem}\label{the:homology_of_Mackey_module}
  Let  $M = (M_*,M^*)$ be a  Mackey $\Lambda\calp$-module.
  Then

  \begin{enumerate}

  \item\label{the_homology_of_Mackey_module:higher_homology}
    $H_n(\calp;M_*)$ vanishes for $n \ge 1$;

    \item\label{the:homology_of_Mackey_module:zeroth_homology}
   We obtain an isomorphism
   \[
   \colim_{\calp} M_* \cong H_0(\calp;M_*) \cong \bigoplus_{\sigma \in \calp} S_{\sigma} M_*;
   \]

   \item\label{the:homology_of_Mackey_module:splitting}
     $S_{\sigma} M_*$ is a direct summand in the $\Lambda$-module $M_*(\sigma)$.

   \end{enumerate}
 \end{theorem}
 \begin{proof}~\ref{the_homology_of_Mackey_module:higher_homology}
   From Lemma~\ref{lem:widehat(s)_bijective}, we obtain an isomorphism
   \[
    \bigoplus_{\sigma \in \calp}  H_n(\calp;E_{\sigma}(S_{\sigma} M_*)) \cong 
   H_n(\calp;\bigoplus_{\sigma \in \calp} E_{\sigma}(S_{\sigma} M_*)) \cong H_n(\calp;M_*).
 \]
 Since the automorphism group of the object $\sigma$ in $\calp$ is the trivial group
 $\{1\}$, we get for any $\Lambda$-module $N$ an isomorphism
 \[
 H_n(\calp;E_{\sigma}(N)) \cong H_n(\{1\};N) \cong
   \begin{cases} \{0\} & \text{if}\; n \ge 1; \\ N & \text{if}\; n = 0. 
   \end{cases}
 \]
 This follows from the adjunction $(E_{\sigma},\RES_{\sigma})$ of~\cite[Lemma~9.31 on page~171]{Lueck(1989)} and the fact that $\RES_{\sigma}(P_*)$ is a projective
 $\Lambda$-resolution of $\Lambda$.
 \\[1mm]\ref{the:homology_of_Mackey_module:zeroth_homology} For every covariant
 $\Lambda\calp$-module $N$, there are canonical $\Lambda$-isomorphisms
 \[
 \colim_{\calp} N \cong \underline{\Lambda} \otimes_{\Lambda\calp} N \cong H_0(\calp;N),
 \]
 where $\underline{\Lambda}$ is the constant $\Lambda\calp$-module with value $\Lambda$.
 This follows from the adjunction between tensor product and the hom-functor
 and the fact that $-\otimes_{\Lambda \calp} N$ is right-exact,
 see~\cite[9.21 and 9.23 on page~169]{Lueck(1989)}.
 \\[1mm]\ref{the:homology_of_Mackey_module:splitting} This follows from
 Lemma~\ref{lem:overline(s)_tau_M_is_a_section}.
\end{proof}

Let $\ch_n(\Sigma,\sigma)$ be the set of $n$-chains $\sigma_0 < \sigma_1 < \cdots < \sigma_n$ in $\calp$ with
$\sigma_0 = \sigma$. Define the integer
\begin{equation*}
n_{\sigma} \coloneqq  \sum_{n \ge 0} (-1)^n \cdot |\ch_n(\Sigma,\sigma)|.
\end{equation*}
  
Fix a class of $\Lambda$-modules $\cale$ with the property that for an exact
sequence $0 \to V_0 \to V_1 \to V_2 \to 0$ the $\Lambda$-module $V_1$ belongs to $\cale$
if and only if both $V_0$ and $V_2$ belong to $\cale$. An example is the class $\cale$ of
$\Lambda$-modules whose underlying set is finite, and for a Noetherian ring $\Lambda$ the
class $\cale$ of finitely generated $\Lambda$-modules. Denote by $G_0(\cale)$ the
Grothendieck group of elements in $\cale$, i.e., the abelian group with the
isomorphism classes of elements in $\cale$ as generators and relations
$[V_1] = [V_0]+ [V_2]$ for every short exact sequence $0 \to V_0 \to V_1 \to V_2 \to 0$ of
$\Lambda$-modules belonging to $\cale$.

\begin{theorem}\label{the:computation_in_G_0(Lambda)}
  Let  $M = (M_*,M^*)$ be a Mackey $\Lambda\calp$-module.
  Suppose that $M_*(\sigma)$ lies in $\cale$ for all $\sigma \in \calp$.

 Then we get in $G_0(\cale)$
 \[
 [\colim_{\calp} M_*] = \sum_{\sigma \in \calp} n_{\sigma} \cdot [M_*(\sigma)].
 \]
\end{theorem}
\begin{proof}
  The bar-resolution yields a finite free $\Lambda\calp$-resolution $C_*$ of the constant
  $\Lambda \calp$-module $\underline{\Lambda}$ with value $\Lambda$ such that 
  \[
    C_n = \bigoplus_{\sigma \in \calp} \bigoplus_{\ch_n(\Sigma,\sigma)} \Lambda\mor_{\calp}({?},\sigma),
   \]
   see~\cite[Section~3]{Davis-Lueck(1998)}. Since $C_*$ is a finite free $\Lambda\calp$-chain complex, the
   $\Lambda$-chain complex $C_* \otimes_{\Lambda\calp} M_*$ is a finite-dimensional
   $\Lambda$-complex whose $\Lambda$-chain modules belong to $\cale$, and we get in $G_0(\cale)$
 \begin{eqnarray*}
   \sum_{n \ge 0} (-1)^n \cdot [H_n(C_* \otimes_{\Lambda \calp} M_*)]
   & = &
   \sum_{n \ge 0} (-1)^n \cdot [C_n \otimes_{\Lambda \calp} M_*]
   \\
   & = &
   \sum_{n \ge 0} (-1)^n \cdot \bigl(\sum_{\sigma \in \calp} \sum_{\ch_n(\Sigma,\sigma)} [M_*(\sigma)]\bigr)
   \\
   & = &
   \sum_{\sigma \in \calp} \bigl(\sum_{n \ge 0} (-1)^n \cdot |\ch_n(\Sigma,\sigma)|\bigr)  \cdot [M_*(\sigma)]
   \\
   & = &
   \sum_{\sigma \in \calp} n_{\sigma} \cdot [M_*(\sigma)].
 \end{eqnarray*}
 Since $H_n(C_* \otimes_{\Lambda \calp} M_*)$ agrees with $H_n(\calp;M_*)$, the claim follows
 from Theorem~\ref{the:homology_of_Mackey_module}.
\end{proof}


\subsection{Proof of Theorem~\ref{the:main_theorem}}
\label{subsec:Proof_of_the_main_Theorem}

We defined the functor $\calh^{\calc}_n \colon \Groups \to \Modcat{\Lambda}$
in~\eqref{calh_upper_calc_n}.   We conclude from
Example~\ref{exa:Mackey_modules_coming_from_graph_products} that
$(\calh^{\calc}_n \circ W_* \circ I,\calh^{\calc}_n \circ W^* \circ I)$ is a
Mackey $\Lambda\calp$-module. Now Theorem~\ref{the:main_theorem} follows from
Theorem~\ref{Computing_calh_upper_W(Sigma)(EGF(W(Sigma))(calc)},
Theorem~\ref{the:homology_of_Mackey_module}, and
Theorem~\ref{the:computation_in_G_0(Lambda)}.


\typeout{---  Brief review of the  Baum--Connes Conjecture  and the Farrell--Jones Conjecture ----}

\section{Isomorphism Conjectures in $K$- and $L$-theory}%
\label{sec:Brief_review_of_the_Baum-Connes_Conjecture_and_the_Farrell-Jones_Conjecture}

In this section we review the Isomorphism Conjectures of Baum--Cones and Farrell--Jones and recollect the most important results on the passage from $\calfin$ to $\calvcyc$ in the Farrell--Jones setting.

Let $\calc$ be a non-empty class of groups which is closed under isomorphisms, passage to
subgroups and passage to quotient groups. Recall that given a group $G$, we denote by $\EGF{G}{\calc}$ the classifying space of $G$ with respect to the family of subgroups
$\calf_{\calc}(G) = \{ H \subseteq G \mid H \in \calc\}$.
Consider a covariant $\Groupoids$-spectrum
\[
\bfE\colon \Groupoids \to \Spectra
\]
which respects equivalences. We obtain an equivariant homology theory $H^?_*(-;\bfE)$ associated to $\bfE$ from Theorem~\ref{the:GROUPOID-spectra_and_equivariant_homology_theories}. 

Then the \emph{Meta-Isomorphism Conjecture for $\bfE$ and
  the class $\calc$} predicts that the projection $\pr \colon \EGF{G}{\calc} \to G/G$
  induces for all $n \in \IZ$ an isomorphism
  \[
    H^G_n(\pr;\bfE) \colon H_n^G(\EGF{G}{\calc};\bfE)
    \xrightarrow{\cong} H_n^G(G/G;\bfE) = \pi_n(\bfE(G)).
  \]

If we make the appropriate choices for $\bfE$ and $\calc$, this
  becomes the Baum--Connes Conjecture or it becomes the Farrell--Jones Conjecture for algebraic $K$-theory, for algebraic $L$-theory, for Waldhausen's $A$-theory, or for topological Hochschild homology.


\subsection{The Baum--Connes Conjecture}\label{subsec:The_Baum_Connes_Conjecture}

Given a discrete group $G$, denote by $C^*_r(G)$ and $C^*_r(G;\IR)$ its \emph{reduced
  complex and reduced real group $C^*$-algebra} and by $C^*_m(G)$ and $C^*_m(G;\IR)$ 
its \emph{maximal complex and maximal real group $C^*$-algebra}. There are covariant
functors
\begin{eqnarray}
  \bfK_{C^*_m} \colon \Groupoids &\to \Spectra;
\label{bfK_(C_upper_ast_m)}
\\                                 
  \bfK_{C^*_m;\IR} \colon \Groupoids &\to \Spectra,
\label{bfK_(C_upper_ast_m_R)}
\end{eqnarray}
which send equivalences of groupoids to weak equivalences of spectra and satisfy
$\pi_n(\bfK_{C^*_m}(G)) = K_n(C^*_m(G))$ and
$\pi_n(\bfK_{C^*_m;\IR}(G)) = KO_n(C^*_m(G;\IR))$ for $n \in \IZ$. Here $K_*$ and $KO_*$ denote
topological $K$-theory. If we consider the class
of finite groups, the Meta-Isomorphism Conjecture
reduces to the \emph{Baum--Connes Conjecture for the maximal group $C^*$-algebra}. It predicts the
bijectivity of the assembly maps for $n \in \IZ$
\begin{eqnarray*}
  K_n^G(\eub{G}) &\to & K_n(C^*_m(G));
  \\
  KO^G_n(\eub{G}) &\to  & KO_n(C^*_m(G;\IR)),
\end{eqnarray*}
where the source is given by equivariant $K$-homology for which we have
the identifications $K_n^G(\eub{G}) = H_n^G(\eub{G};\bfK_{C^*_m})$ and
$KO_n^G(\eub{G}) = H_n^G(\eub{G};\bfK_{C^*_m;\IR})$.

We can apply $K$-theory to the natural maps of $C^*$-algebras $C^*_m(G) \to C^*_r(G)$ and $C^*_m(G;\IR) \to C^*_r(G;\IR)$
to obtain maps $f_n \colon  K_n(C^*_m(G)) \to K_n(C^*_r(G))$ and
$f_n^{\IR} \colon  KO_n(C^*_m(G;\IR)) \to KO_n(C^*_r(G;\IR))$.
The \emph{Baum--Connes Conjecture} predicts that the composites
\begin{eqnarray*}
& K_n^G(\eub{G}) \to  K_n(C^*_m(G)) \xrightarrow{f_n} K_n(C^*_r(G));&
  \\
&  KO^G_n(\eub{G})  \to KO_n(C^*_m(G;\IR)) \xrightarrow{f_n^{\IR}} KO_n(C^*_r(G;\IR)) &
\end{eqnarray*}
are bijective for all $n \in \IZ$.

There are counterexamples to the Baum--Connes Conjecture for the maximal group
$C^*$-algebra, but no counterexamples to the Baum--Connes Conjecture are known. We want to
consider the Baum--Connes Conjecture for the maximal group $C^*$-algebra since
$K_n(C^*_m(G))$ and $KO_n(C^*_m(G;\IR))$ are functorial in $G$ for all group homomorphisms,
whereas $K_n(C^*_r(G))$ and $KO_n(C^*_r(G;\IR))$ are functorial for injective group
homomorphisms, but not in general for any group homomorphism. Moreover the covariant
functors~\eqref {bfK_(C_upper_ast_m)} and~\eqref{bfK_(C_upper_ast_m_R)} are defined on
$\Groupoids$. This ensures that the induction
structure is available for all group homomorphisms and not only for injective group homomorphisms as it is the case if we replace~\eqref {bfK_(C_upper_ast_m)}
and~\eqref{bfK_(C_upper_ast_m_R)} by their versions for the reduced $C^*$-algebras. We
later want to apply the induction structure also to certain split surjective group homomorphisms, 
see Remark~\ref{rem:retratcions}.

There is a more general \emph{Baum--Connes Conjecture with coefficients}, which is known to be true for a large class of groups and which has good inheritance properties. In particular, the class of groups satisfying the Baum--Connes Conjecture with coefficients is closed under taking graph products, since it is stable under finite direct products and amalgamated products, see~\cite{Oyono-Oyono(2001)} and~\cite{Oyono-Oyono(2001b)}.


\subsection{The Farrell--Jones Conjecture}\label{subsec:The_Farrell-Jones_Conjecture}

Given a ring $R$ (with involution), there are covariant functors
\begin{eqnarray*}
\bfK_R \colon \Groupoids &\to & \Spectra;
\\
 \bfL_R^{\langle -\infty \rangle} \colon \Groupoids &\to & \Spectra,
\end{eqnarray*}
which send equivalences of groupoids to weak equivalences of spectra and satisfy
$\pi_n(\bfK_R(G)) = K_n(RG)$ and $\pi_n(\bfL_R^{\langle -\infty \rangle}(G)) = L_n^{\langle -\infty \rangle}(RG)$.
Here $K_*$ denotes non-connective algebraic $K$-theory and $L_*^{\langle -\infty \rangle}$ denotes algebraic $L$-theory
with decoration $\langle -\infty \rangle$.
If we consider the class of virtually cyclic groups,
the Meta-Isomorphism Conjecture
reduces to the \emph{$K$-theoretic or the $L$-theoretic Farrell--Jones Conjecture} which
predicts that for all $n \in \IZ$ the corresponding map
\begin{eqnarray*}
  H^G_n(\edub{G};\bfK_R)
  &\to &
   H^G_n(G/G;\bfK_R)  = K_n(RG);
  \\
  H^G_n(\edub{G};\bfL_R^{\langle -\infty \rangle})
  &\to &
  H^G_n(G/G;\bfL_R^{\langle -\infty \rangle}) = L_n^{\langle -\infty \rangle}(RG)
\end{eqnarray*}
is bijective.

There is a more general \emph{Full Farrell--Jones Conjecture} which allows additive $G$-categories as coefficients. It is known to be true for a large class of groups and has good inheritance properties. In particular, the class of groups satisfying the Full Farrell--Jones Conjecture is closed under taking graph products, which is a result of Gandini--R\"uping~\cite{Gandini-Rueping(2013)}. There also is a version of the Farrell--Jones Conjecture for Waldhausen's $A$-theory which we will not discuss here. It satisfies similar inheritance properties as the Full Farrell--Jones Conjecture, see~\cite{Enkelmann-Lueck-Malte-Ullmann-Winges(2018)} and~\cite{Ullmann-Winges(2019)}. Also the following Theorem~\ref{the:passage_from_calfin_to_calvycy} \ref{the:passage_from_calfin_to_calvycy:split_injective} holds in this setting, see~\cite{Bunke-Kasprowski-Winges(2018)}.


\subsection{The passage from $\calfin$ to $\calvcyc$}%
\label{subsec:The_passage_from_calfin_to_calvcyc}

The Farrell--Jones Conjecture is more complicated than  the Baum--Connes Conjecture
since for the Farrell--Jones Conjecture the class of virtually cyclic groups has to be considered,
whereas for the Baum--Connes Conjecture the class of finite groups suffices.
Hence one has to understand the passage from $\eub{G}$ to $\edub{G}$ in the Farrell--Jones setting.

\begin{theorem}\label{the:passage_from_calfin_to_calvycy}\

  \begin{enumerate}
  \item\label{the:passage_from_calfin_to_calvycy:split_injective}
    Let $G$ be any group and $R$ be a ring. Then the relative assembly maps
    \begin{eqnarray*}
  H^G_n(\eub{G};\bfK_R)
  &\to &
   H^G_n(\edub{G};\bfK_R);
  \\
  H^G_n(\eub{G};\bfL_R^{\langle -\infty \rangle})
  &\to &
  H^G_n(\edub{G};\bfL_R^{\langle -\infty \rangle})
\end{eqnarray*}
   are split injective for all $n \in \IZ$; 
    
  \item\label{the:passage_from_calfin_to_calvycy:general_K-theory_R_reg}
  Let $G$ be any group and $R$ be a regular ring. Then the  relative assembly map
  \[
  H_n^G(\eub{G};\bfK_R) \to H_n^G(\edub{G};\bfK_R)
  \]
  is rationally bijective  for all $n \in \IZ$; 

  \item\label{the:passage_from_calfin_to_calvycy:general_K-theory_R_reg_Q_subseteq_R}
    Let $G$ be any group and $R$ be a regular ring. Suppose that for any finite subgroup $H \subseteq G$
    its order $|H|$ is invertible in $R$. Then the  relative assembly map
  \[
  H_n^G(\eub{G};\bfK_R) \to H_n^G(\edub{G};\bfK_R)
  \]
  is bijective for all $n \in \IZ$; 

  \item\label{the:passage_from_calfin_to_calvycy:special_K-theory}
    Let $R$ be a regular ring. Let $W = W(X,\calw)$ be a graph product and $d$ be a natural number. Suppose that for any vertex $v \in V$ the group $W_v$ is either torsionfree or a finite group whose order divides $d$.
Then the relative assembly map
  \[
  H_n^W(\eub{W};\bfK_R) \to H_n^W(\edub{W};\bfK_R)
  \]
  is bijective after inverting $d$  for all $n \in \IZ$; 

  \item\label{the:passage_from_calfin_to_calvycy:general_L-theory}
  Let $G$ be any group and $R$ be a ring with involution. Then the  relative assembly map
  \[
  H_n^G(\eub{G};\bfL^{\langle -\infty \rangle}_R) \to H_n^G(\edub{G};\bfL^{\langle -\infty \rangle}_R)
  \]
  is bijective after inverting $2$  for all $n \in \IZ$; 

\item\label{the:passage_from_calfin_to_calvycy:general_L-theory_1/2_in_R}
  Let $G$ be any group and $R$ be a ring with involution such that $2$ is invertible in $R$.
  Then the  relative assembly map
  \[
  H_n^G(\eub{G};\bfL^{\langle -\infty \rangle}_R) \to H_n^G(\edub{G};\bfL^{\langle -\infty \rangle}_R)
  \]
  is bijective for all $n \in \IZ$.

  \item\label{the:passage_from_calfin_to_calvycy:special_L-theory}
  Let $R$ be a ring with involution.
  Let $W = W(X,\calw)$ be a graph product. Suppose that for any vertex $v \in V$
  the group $W_v$ is either torsionfree  or a finite  group of odd order.
  Then the  relative assembly map
  \[
  H_n^W(\eub{W};\bfL^{\langle -\infty \rangle}_R) \to H_n^W(\edub{W};\bfL^{\langle -\infty \rangle}_R)
  \]
  is bijective for all $n \in \IZ$.
\end{enumerate}
\end{theorem}
\begin{proof}~\ref{the:passage_from_calfin_to_calvycy:split_injective}
  See~\cite{Bartels(2003b)} and~\cite{Lueck-Steimle(2016splitasmb)}.
  \\[1mm]~\ref{the:passage_from_calfin_to_calvycy:general_K-theory_R_reg}
  This is proved in~\cite[Theorem~0.3]{Lueck-Steimle(2016splitasmb)}.
  \\[1mm]~\ref{the:passage_from_calfin_to_calvycy:general_K-theory_R_reg_Q_subseteq_R}
  See~\cite[Proposition~70 on page~744]{Lueck-Reich(2005)}.
  \\[1mm]~\ref{the:passage_from_calfin_to_calvycy:special_K-theory}
  Let $\calvcyc_I$ be the class of virtually cyclic groups of type I, i.e., groups which admit a homomorphism to $\IZ$ with finite kernel.
  Then the
  relative assembly map
  \[
  H_n^W(\EGF{W}{\calvcyc_I};\bfK_R) \to H_n^W(\EGF{W}{\calvcyc};\bfK_R)
\]
is bijective for all $n \in \IZ$ by~\cite[Theorem~1.1]{Davis-Quinn-Reich(2011)}.
Hence it suffices to show that
\[
H_n^W(\EGF{W}{\calfin};\bfK_R) \to H_n^W(\EGF{W}{\calvcyc_I};\bfK_R)
\]
is bijective for all $n \in \IZ$. By the same argument as it appears in~\cite[Proof of
Theorem~0.3 on page~370]{Lueck-Steimle(2016splitasmb)} the claim is reduced to showing that
for any infinite virtually cyclic subgroup $H \subseteq W$ of type $I$ we get
\[
N\!K_n(RK_H;\phi)[1/d] = 0,
\]
where $K_H \subseteq H$ is a finite subgroup such that $H/K_H$ is infinite cyclic and
$\phi \colon K_H \to K_H$ is given by conjugation with an element $h \in H$ which is sent under the projection $H \to H/K_H$ to a generator.
Any finite subgroup of $W$ is conjugated into a group $W(\sigma)$ for some simplex
 $\sigma$ of $\Sigma$ such that $W_v$ is finite for every $v \in V \cap \sigma$. 
(Indeed, this is obvious if $\Sigma$ is a simplex. Otherwise, we can express $W$ as an amalgamated product and use the fact that a finite subgroup of an amalgamated product is conjugated into one of the factors, see~\cite[Theorem~8 in~4.3 on page~36]{Serre(1980)}. Then conclude via induction on the number of vertices.)
Hence we can assume that there exists a
simplex $\sigma$ of $\Sigma$ such that $K_H \subseteq W(\sigma)$ and $W(\sigma)$ is finite.
There is a group homomorphism $r \colon W \to W(\sigma)$ whose restriction to $W(\sigma)$ is the identity,
see Remark~\ref{rem:retratcions}. Consider
$w \in K_H$. Then $hwh^{-1}$ belongs to $K_H$ again. 
We compute
	\[\phi(w) = hwh^{-1} = r(hwh^{-1}) = r(h)r(w)r(h)^{-1}=r(h)wr(h)^{-1}.
	\]
	Hence $\phi$ is given by conjugation with $r(h)\in W(\sigma)$. The order of $r(h)$ and thus also of $\phi$ divides $d$. We conclude
	from~\cite[Theorem~9.4]{Lueck-Steimle(2016splitasmb)} that $N\!K_n(RK_H;\phi)[1/d] = 0$.
 \\[1mm]~\ref{the:passage_from_calfin_to_calvycy:general_L-theory} This is proved
 in~\cite[Lemma~4.2]{Lueck(2005heis)}.
 \\[1mm]~\ref{the:passage_from_calfin_to_calvycy:general_L-theory_1/2_in_R}
 The proof is analogous to the first part of the proof of~\ref{the:passage_from_calfin_to_calvycy:special_K-theory}
 using the fact that UNil-groups vanish, if $1/2$ is contained in $R$, see~\cite[Corollary~3]{Cappell(1974c)}
 and that the map
 \[H_n^G(\EGF{G}{\calfin};\bfL_R^{\langle -\infty \rangle}) \to H_n^G(\EGF{G}{\calvcyc_I};\bfL_R^{\langle -\infty \rangle})
 \]
 is bijective for all $n \in \IZ$, see~\cite[Lemma~4.2]{Lueck(2005heis)}.
 \\[1mm]~\ref{the:passage_from_calfin_to_calvycy:special_L-theory} Any finite
 subgroup of $W$ is conjugated into a group $W(\sigma)$ for some simplex
 $\sigma$ of $\Sigma$ such that $W_v$ is finite for every $v \in V \cap \sigma$. 
 Hence every finite subgroup has odd order and thus every infinite
 virtually cyclic subgroup of $W$ is of type I.  Now the claim follows
 from~\cite[Lemma~4.2]{Lueck(2005heis)}.  This finishes the proof of
 Theorem~\ref{the:passage_from_calfin_to_calvycy}.
\end{proof}


\typeout{-------------  Right-angled Artin groups ----------------------}

\section{Right-angled Artin groups}
\label{sec:right-angled_Artin_groups}

In this section we want to compute the group homology, the algebraic $K$- and $L$-theory,
and the topological $K$-theory of a right-angled Artin group $W$.  Recall that a
\emph{right-angled Artin group} is a graph product $W = W(X,\calw)$ for which each of the
groups $W_v$ is infinite cyclic. Note that $W$ is torsionfree. Right-angled Artin groups satisfy the Baum--Connes Conjecture and the Baum--Connes Conjecture for the maximal group $C^*$-algebra, which follows from~\cite{Davis-Januszkiewicz(2000)} and~\cite{Higson-Kasparov(2001)}. Both the $K$-theoretic and the $L$-theoretic Farrell--Jones Conjecture are
satisfied for right-angled Artin groups,
see~\cite{Bartels-Lueck(2012annals)} and~\cite{Wegner(2012)}.
For general information
about right-angled Artin groups we refer for instance to Charney~\cite{Charney(2007)}.

In the sequel we denote by $r_k$ the number of $k$-simplices in $\calp = \calp(\Sigma)$ and put $r = |\calp| = \sum_{k = -1}^{\dim(\Sigma)} r_k$.
Recall that the empty simplex is allowed in $\calp$ and has dimension $-1$.


Let $\calk_*$ be a (non-equi\-va\-riant) generalized homology theory with values in
$\Lambda$-modules.  Let $X$ be a $CW$-complex. 
It follows from the axioms of a generalized homology theory that there is an isomorphism, natural in $X$
\begin{equation*}
  B_n(X) :=\calk_n(\pr) \times \bigl(s_n \circ \calk_n(\id_X \times i)\bigr) \colon \calk_n(X \times S^1)
  \xrightarrow{\cong} \calk_n(X) \times \calk_{n-1}(X),
\end{equation*}
where we denote by $\pr \colon X \times S^1 \to X$ the projection, by $i \colon S^1 = (S^1,\emptyset) \to (S^1,\pt)$ the inclusion, and by $s_n \colon \calk_n(X \times (S^1,\{\pt\})) \xrightarrow{\cong} \calk_{n-1}(X)$ the suspension isomorphism.

By induction over $k \ge 0$ we obtain an isomorphism
\begin{equation}
  B_n^k \colon \calk_n(T^k) \xrightarrow{\cong} \prod_{i = 0}^k\;
\prod_{j = 1}^{\binom{k}{i}} \calk_{n-i}(\pt), 
  \label{homological_Bass-Heller_Swan}
\end{equation}where we denote by $T^k$ the $k$-dimensional torus $\prod_{i=1}^k S^1$. Note that $T^0 = \pt$.
For $i =1,2, \ldots, k$, let $T^k_i \subseteq T^k$ be the subspace consisting of elements $(z_1,z_2, \ldots, z_k)$ with $z_i = \ast$, where $\ast$ is a fixed base point in $S^1$. Let $j_i^k \colon T^k_i \to T^k$ be the inclusion.
We will identify $T_k^k = T^{k-1}$. 

\begin{lemma}\label{lem:top_component}
  For every $k \ge 0$ and $n \in \IZ$ there is an isomorphism 
  \[
   c_n^k \colon \cok\left(\bigoplus_{i = 1}^k \calk_n(j^k_i) \colon \bigoplus_{i = 1}^k \calk_n(T^k_i)
    \to  \calk_n(T^k) \right)
    \xrightarrow{\cong} \calk_{n-k}(\pt).
  \]
  Its inverse is induced by the restriction of the inverse of the isomorphism $B_n^k$
  of~\eqref{homological_Bass-Heller_Swan} to the factor $\calk_{n-k}(\pt)$ for the index $i = k$.
\end{lemma}
\begin{proof} We use induction over $k = 0,1,2, \ldots$. If $k = 0$, take
  $c_n^0 = \id_{\calk_{n}(\pt)}$.  The induction step from $(k-1)$ to $k \ge 1$ is done as
  follows.  We have the following commutative diagram of $\Lambda$-modules
  \[
    \xymatrix@!C=18em{\bigoplus_{i =1}^k \calk_n(T^k_i)
      \ar[r]^-{\bigoplus_{i =1}^k \calk_n(j^k_i)} \ar[d]_{\id}^{\cong}
    &
    \calk_n(T^k) \ar[dd]^{\id}_{\cong}
    \\
    \calk_n(T^k_k) \oplus \bigoplus_{i =1}^{k-1} \calk_n(T^k_i)
    \ar[d]_{\id \oplus \bigoplus_{i =1}^{k-1} \id}^{\cong}
    &
    \\
    \calk_n(T^{k-1}) \oplus \bigoplus_{i =1}^{k-1} \calk_n(T^{k-1}_i \times S^1)
    \ar[d]_{\id \oplus \bigoplus_{i =1}^{k-1} B_n(T^{k-1}_i)}^{\cong}
    &
    \calk_n(T^{k-1} \times S^1) \ar[dd]^{B_n(T^{k-1})}_{\cong}
    \\
    \calk_n(T^{k-1}) \oplus \bigoplus_{i =1}^{k-1} \bigl(\calk_n(T^{k-1}_i)
    \oplus \calk_{n-1}(T^{k-1}_i)\bigr) \ar[d]_-{f}^-{\cong}
    &
    \\
    {\begin{array}{c}
    \left(\calk_n(T^{k-1}) \oplus \bigoplus_{i =1}^{k-1} \calk_n(T^{k-1}_i)\right)
    \\
    \oplus 
    \\
    \bigoplus_{i =1}^{k-1}  \calk_{n-1}(T^{k-1}_i)
    \end{array}}
    \ar[r]_-{\begin{pmatrix} f_0 & 0\\ 0 & f_1 \end{pmatrix}}
    &
    {\begin{array}{c}
    \calk_n(T^{k-1})
    \\
    \oplus
    \\
    \calk_{n-1}(T^{k-1})
    \end{array}}
  }
    \]
    where $f$ is the obvious isomorphism,
    $f_0 \coloneqq \id_{\calk_n(T^{k-1})} \oplus \bigoplus_{i =1}^{k-1}\calk_n(
      j_i^{k-1})$, and
    $f_1 \coloneqq \bigoplus_{i =1}^{k-1} \calk_{n-1}(j^{k-1}_i)$. Since $f_0$ is surjective, the
    diagram above induces an isomorphism
    \[\cok\left(\bigoplus_{i =1}^k \calk_n(j^k_i)\right)
      \xrightarrow{\cong} \cok(f_0) \oplus \cok(f_1)
      = \{0\} \oplus \cok\left(\bigoplus_{i =1}^{k-1} \calk_{n-1}(j^{k-1}_i)\right).
   \]
   By induction hypothesis we have an isomorphism
   \[
     c^{k-1}_{n-1}\colon \cok\left(\bigoplus_{i =1}^{k-1} \calk_{n-1}(j^{k-1}_i)\right)
     \xrightarrow{\cong} \calk_{n-k}(\pt).
   \]
   This finishes the induction step and hence the proof of
   Lemma~\ref{lem:top_component}.
 \end{proof}


\subsection{Group homology}
\label{subsec:Group_homology_of_right-angled_Artin_groups}

Let $\calk_*$ be any generalized homology theory with values in $\Lambda$-modules. Notice
that for any group $G$ the $CW$-complex $EG \times_G \eub{G}$ is a model for $BG$ since $\eub{G}$ is contractible after
forgetting the $G$-action.  We have introduced the equivariant homology theory
given by the Borel construction and $\calk_*$ in
Example~\ref{exa:Borel_homology}. We conclude from
Theorem~\ref{the:main_theorem} and Lemma~\ref{lem:top_component} that there is an explicit
$\Lambda$-isomorphism
\[
\bigoplus_{ \sigma \in \calp} \calk_{n- \dim(\sigma)-1}(\pt) \xrightarrow{\cong} \calk_n(BW).
\]

If we take for $\calk_*$ singular homology $H_*(-;\Lambda)$ with coefficients in $\Lambda$,
this boils down to the well-known, see for example~\cite[Corollary~11]{Kim-Roush(1980)}, isomorphism of $\Lambda$-modules
\[
 \Lambda^{r_{n-1}} \xrightarrow{\cong} H_n(BW;\Lambda).
\]
In particular we get the following relation for  the Euler characteristics
\[
\chi(BW) = 1- \chi(\Sigma).
\]


\subsection{Algebraic $K$-theory}
\label{subsec:Algebraic_K-theory_of_right-angled_Artin_groups}
Let $R$ be a regular ring. 
We conclude from
Theorem~\ref{the:main_theorem},
Theorem~\ref{the:passage_from_calfin_to_calvycy}~%
\ref{the:passage_from_calfin_to_calvycy:special_K-theory}, and Lemma~\ref{lem:top_component}
that there is an explicit isomorphism of abelian groups
\[
\bigoplus_{ \sigma \in \calp} K_{n- \dim(\sigma) - 1}(R) \xrightarrow{\cong} K_n(RW).
\]

Its restriction to the summand belonging to $\sigma$ is the composite of the map
$K_n(RW(\sigma)) \to K_n(RW)$ coming from the inclusion $\IZ^{\dim(\sigma)+1} = W(\sigma) \to W$
with the restriction of the inverse of the iterated Bass-Heller-Swan isomorphism
\[
\bigoplus_{i = 0}^{\dim(\sigma)+1} \bigoplus_{j = 1}^{\binom{\dim(\sigma)+1}{i}} K_{n-i}(R) 
\xrightarrow{\cong} K_n(R[\IZ^{\dim(\sigma)+1}])
\]
to the summand $K_{n- \dim(\sigma)-1}(R)$ belonging to $i = \dim(\sigma)+1$.

Since for a regular ring $R$ its negative $K$-theory vanishes, we conclude
$K_n(RW) = 0$ for $n \le -1$. If we take $R = \IZ$, we conclude
that $K_n(\IZ W)$ for $n \le -1$, $\widetilde{K}_0(\IZ W)$, and $\Wh(W)$ vanish what is actually true
if we replace $W$ by any torsionfree group satisfying the Farrell--Jones Conjecture.


\subsection{Algebraic $L$-theory}
\label{subsec:Algebraic_L-theory_of_right-angled_Artin_groups}

Let $R$ be a ring with involution.  We conclude from Theorem~\ref{the:main_theorem},
Theorem~\ref{the:passage_from_calfin_to_calvycy}~%
\ref{the:passage_from_calfin_to_calvycy:special_L-theory}, and
Lemma~\ref{lem:top_component} that there is an explicit isomorphism of abelian groups
\[
  \bigoplus_{ \sigma \in \calp} L_{n- \dim(\sigma)-1}^{\langle -\infty \rangle}(R)
  \xrightarrow{\cong} L_n^{\langle -\infty \rangle}(RW).
\]

Its restriction to a summand comes from the Shaneson splitting.


\subsection{Topological  $K$-theory}
\label{subsec:Topological_K-theory_of_right-angled_Artin_groups}

 We conclude from Theorem~\ref{the:main_theorem}  and
Lemma~\ref{lem:top_component} that there are explicit isomorphisms of abelian groups
\begin{eqnarray*}
  \bigoplus_{\sigma \in \calp} K_{n-\dim(\sigma)-1}(\IC)
  & \xrightarrow{\cong} &
  K_n(C^*_m(W)) \cong K_n(C^*_r(W)); 
  \\
   \bigoplus_{\sigma \in \calp} KO_{n-\dim(\sigma)-1}(\IR)
  & \xrightarrow{\cong} &
   KO_n(C^*_m(W;\IR)) \cong KO_n(C^*_r(W;\IR)).
\end{eqnarray*}                      
In particular we get an isomorphism of abelian groups
\[
  K_n(C^*_m(W)) \cong K_n(C^*_r(W)) \cong \IZ^{t_n},
\]
if we put $t_n= \sum_{\substack{k \in \{-1,0,1,2,\ldots, \dim(\Sigma)\}\\ (n-k) \;\text{odd}}} r_k$.


\typeout{-------------  Right-angled  Coxeter groups ----------------------}

\section{Right-angled Coxeter groups}
\label{sec:right-angled_Coxeter_groups}

In this section we want to compute the group homology, the algebraic $K$- and $L$-theory,
and the topological $K$-theory of a right-angled Coxeter group $W$.  Recall that a
\emph{right-angled Coxeter group} is a graph product $W = W(X,\calw)$ for which each of the
groups $W_v$ is cyclic of order two. 
Right-angled Coxeter groups satisfy the Baum--Connes Conjecture and the Baum--Connes Conjecture for the maximal group $C^*$-algebra, which follows from~\cite{Davis-Januszkiewicz(2000)} and~\cite{Higson-Kasparov(2001)}.
Both the $K$-theoretic and the $L$-theoretic Farrell--Jones Conjecture are
satisfied for right-angled Coxeter groups,
see~\cite{Bartels-Lueck(2012annals)} and~\cite{Wegner(2012)}.

In the sequel we denote by $r_k$ the number of $k$-simplices in $\calp=\calp(\Sigma)$ and put
$r = |\calp| = \sum_{k = -1}^{\dim(\Sigma)} r_k$.  Recall that the empty simplex is
allowed in $\calp$ and has dimension $-1$.

During this section we denote by $C_2$ the cyclic group of order two.
Fix an integer $k \ge 1$. We will identify $C_2^k = C_2^{k-1} \times C_2$ and put $C_2^0=\{1\}$. For $i=1,2,\ldots,k$, let $(C_2^k)_i$ be the subgroup of $C_2^k$ consisting of those elements $(a_1,a_2,\ldots,a_k)$ satisfying $a_i=0$ and denote by $j_i^k\colon (C_2^k)_i\to C_2^k$ the inclusion.


\subsection{Group homology}
\label{subsec:Group_homology}

Define for $n \ge 0$ and $k \ge 0$ an integer
\begin{equation*}
\rho_{n,k} \coloneqq \sum_{j = k}^{n-1} (-1)^{n-1-j} \cdot \binom{j}{k},
\end{equation*}
where here and in the sequel we use the convention  $\sum_{j = a}^b c_j = 0$ for $a > b$. 
\begin{theorem}\label{the:group_homology_Coxeter_groups}
  We have for $n \ge 1$
  \[
  H_n(W;\IZ) \cong \bigoplus_{\sigma \in \calp} \bigoplus_{j=1}^{\rho_{n,\dim(\sigma)}}C_2.
  \]
\end{theorem}

Its proof needs some preparation. Firstly, the numbers $\rho_{n,k} $ satisfy the following.

\begin{eqnarray}
\rho_{n,0}
& = &
\begin{cases}
   1 &  \quad \text{for}\; n \ge 1,  n\; \text{odd};
   \\
   0 &  \quad \text{for}\; n \ge 0,  n\; \text{even};
  \end{cases}
\label{rho_(n,0)}
  \\
  \rho_{n,k}
  & = &
 \rho_{n-1,k-1} + \rho_{n-1,k} \quad \text{for}\; k,n \ge 1;
  \label{rho_(n,k)_two_summands}
  \\
  \rho_{n,k}
 & = &
\sum_{i= 1}^{n-1} \rho_{i,k-1}\quad  \text{for}\; n \ge 0 \; \text{and} \; k \ge 1.
\label{rho_(n,k)_total}   
\end{eqnarray}
Equation~\eqref{rho_(n,0)} follows directly from the definition and equation~\eqref{rho_(n,k)_two_summands} follows from an easy calculation. Then equation~\eqref{rho_(n,k)_total} follows by induction from equation~\eqref{rho_(n,k)_two_summands}.

\begin{lemma}\label{lem:group_homology_C2k}
	We have for $n \ge 1$ and $k \ge 0$
	\[
		H_n(C_2^k;\IZ) \cong \bigoplus_{i=1}^{t_{n,k}} C_2
	\]
	with $t_{n,k} = \sum_{j=1}^k \binom{k}{j} \cdot \rho_{n,j-1}$.
\end{lemma}
\begin{proof}
	The assertion is obviously true for $k = 0$.
	The induction step from $k-1$ to $k \ge 1$ is done as follows.
	Recall that $H_n(C_2;\IZ)$ is $\IZ$ if $n = 0$, $C_2$ if $n \ge 1 $ and $n$ is odd, and $\{0\}$ otherwise. 
	The K\"unneth formula gives the following short exact sequence of $\IZ$-modules, which is natural in $C_2^{k-1}$
	\begin{multline*}
  		0 \to \bigoplus_{i+j = n} H_i(C_2^{k-1};\IZ) \otimes_{\IZ} H_j(C_2;\IZ) \to H_n(C_2^{k-1} \times C_2;\IZ)
  		\\
  		\to \bigoplus_{i+j = n-1} \Tor_1^{\IZ}(H_i(C_2^{k-1};\IZ),H_j(C_2;\IZ)) \to  0.
	\end{multline*}
	It splits but the spitting is not natural in $C_2^{k-1}$.
	By rearranging the summands we obtain an isomorphism of $\IZ$-modules
	\[
		H_n(C_2^k;\IZ) \cong H_n(C_2;\IZ) \oplus \bigoplus_{i = 1}^n H_i(C_2^{k-1};\IZ).
	\]
	Using the induction hypothesis we calculate
	\begin{eqnarray*}
 		t_{n,k}
  		& = &
  		t_{n,1} + \sum_{i = 1}^n t_{i,k-1}
  		\\
		& \stackrel{\eqref{rho_(n,0)}}{=}  &  		
		\rho_{n,0} + \sum_{i = 1}^n \sum_{j = 1}^{k-1}  \binom{k-1}{j} \cdot \rho_{i,j-1}
  \\
  & = &
     \rho_{n,0} + \sum_{j = 1}^{k-1} \binom{k-1}{j} \cdot \sum_{i = 1}^{n-1}  \rho_{i,j-1} + \sum_{j = 1}^{k-1} \binom{k-1}{j} \cdot \rho_{n,j-1}
  \\
  & \stackrel{\eqref{rho_(n,k)_total}}{=} &
  \rho_{n,0} + \sum_{j = 1}^{k-1} \binom{k-1}{j} \cdot \rho_{n,j} + \sum_{j = 1}^{k-1} \binom{k-1}{j} \cdot \rho_{n,j-1}
  \\
  & = &
  \rho_{n,0} + \sum_{j = 2}^{k} \binom{k-1}{j-1} \cdot \rho_{n,j-1} + \sum_{j = 1}^{k-1} \binom{k-1}{j} \cdot  \rho_{n,j-1} 
  \\
  & = &
  \sum_{j = 1}^{k} \binom{k-1}{j-1} \cdot \rho_{n,j-1} + \sum_{j = 1}^{k} \binom{k-1}{j} \cdot \rho_{n,j-1}
  \\  
  & = &
  \sum_{j = 1}^{k} \left(\binom{k-1}{j-1}  + \binom{k-1}{j}\right) \cdot \rho_{n,j-1}
  \\
  & = &
 \sum_{j = 1}^{k}   \binom{k}{j} \cdot \rho_{n,j-1}.
\end{eqnarray*}
This finishes the proof of Lemma~\ref{lem:group_homology_C2k}.
\end{proof}

For $n \ge 1$ and $k \ge 0$ define
\[
	S_kH_n \coloneqq \cok \left( \bigoplus_{i=1}^k H_n(j_i^k;\IZ)\colon \bigoplus_{i=1}^k H_n((C_2^k)_i;\IZ) \to H_n(C_2^k;\IZ) \right).
\]
Let the integer $s_{n,k} \ge 0$ be defined by
\[
	S_kH_n \cong \bigoplus_{j=1}^{s_{n,k}} C_2.
\]

\begin{lemma}\label{lem:s_versus_rho}
	For $n \ge 1$ and $k \ge 1$ we have
	\[
		s_{n,k} = \rho_{n,k-1}
	\]
	and $s_{n,0} = 0$.
\end{lemma}
\begin{proof}
	Since $H_n(\{1\};\IZ)=\{0\}$ for $n \ge 1$, we have $s_{n,0}=0$.
	The induction step from $k-1$ to $k\ge 1$ is done as follows.
	Theorem~\ref{the:main_theorem} yields an isomorphism
	\[
		\bigoplus_{j=0}^k \bigoplus_{i=1}^{\binom{k}{j}} S_jH_n \cong H_n(C_2^k;\IZ).
	\]
	Using the induction hypothesis and Lemma~\ref{lem:group_homology_C2k} we conclude
	\begin{eqnarray*}
		s_{n,k}
		& = &
		t_{n,k} - \sum_{j=0}^{k-1} \binom{k}{j} \cdot s_{n,j}
		\\
		& = &
		\sum_{j=1}^k \binom{k}{j} \rho_{n,j-1} - \sum_{j=1}^{k-1} \binom{k}{j} \cdot \rho_{n,j-1}
		\\
		& = &
		\rho_{n,k-1}.
	\end{eqnarray*}
	This finishes the proof of Lemma~\ref{lem:s_versus_rho}.
\end{proof}

Now Theorem~\ref{the:group_homology_Coxeter_groups} follows from
Theorem~\ref{the:main_theorem} applied to the equivariant homology theory given by Borel homology and singular homology with $\IZ$-coefficients, see
Example~\ref{exa:Borel_homology}, and from
Lemma~\ref{lem:s_versus_rho}. Here we use the fact that for any group $G$ the space
$EG \times_G \eub{G}$ is a model for $BG$.

\begin{remark}\label{rem:replacing_Z/2_byZ/2_upperl}
  If we replace in this subsection $C_2 = \IZ/2$ everywhere by $\IZ/p^l$ for some prime
  number $p$ and some natural number $l$, then Theorem~\ref{the:group_homology_Coxeter_groups}
  remains true. This follows from  two facts.  Since $\IZ/p^l$ is a local ring, we conclude
  from~\cite[Lemma~1.2 on page~5]{Milnor(1971)} that
  for every natural number $a$, every summand of the abelian
group $\bigoplus_{i = 1}^a \IZ/p^l$ is isomorphic to $\bigoplus_{j = 1}^b \IZ/p^l$ for some
natural number $b$. The group homology $H_n(\IZ/p^l;\IZ)$ is isomorphic to $\IZ/p^l$ if $n$ is
odd and vanishes for even $n$ with $n \ge 2$.
\end{remark}

\subsection{Negative $K$-groups for $R = \IZ$}
\label{subsec:negative_K_groups_for_Coxeter_groups}

\begin{theorem}
	\label{the:negative_alg_K-theory_of_Coxeter_groups}
  We have $K_n(\IZ W) =\{0\}$ for $n \le -1$.
\end{theorem}
\begin{proof}
  Since right-angled Coxeter groups satisfy the Farrell--Jones Conjecture, we get $K_n(\IZ W) = 0$ for $n \le -2$ and an isomorphism
  \[
    \colim_{H \in \SubGF{W}{\calfin}} K_{-1}(\IZ H) \to K_{-1}(\IZ W)
  \]
  from~\cite[page~749]{Lueck-Reich(2005)}.
 
  Since any finite subgroup of $W$ is isomorphic to $(\IZ/2)^k$ for some natural number $k$,
  and $K_{-1}(\IZ A) = 0$ holds for a finite abelian group whose order is a prime power,
  see~\cite[Theorem~10.6 on page~695]{Bass(1968)} or~\cite{Carter(1980)}, the claim
  follows.
\end{proof}


\subsection{Projective class group  for $R = \IZ$}
\label{subsec:Projective_class_group_for_Coxeter_groups}

\begin{theorem}\label{the:projective_class_groups_Coxeter_groups}\
 \begin{enumerate}
 \item\label{lem:Projective_class_group_for_Coxeter_groups:eub(W)}
 There is an isomorphism
 \[
 	K_0(\IZ) \oplus \bigoplus_{\sigma \in \calp}\IZ/(2^{\dim \sigma-1}) \xrightarrow{\cong} H_0^W(\eub{W};\bfK_{\IZ});
 \]

\item\label{lem:Projective_class_group_for_Coxeter_groups:Nil}
The map
\[H_0^W(\eub{W};\bfK_{\IZ}) \to H_0^W(\edub{W};\bfK_{\IZ})
\]
is an isomorphism after inverting $2$;

\item\label{lem:Projective_class_group_for_Coxeter_groups:FJC}
  The canonical map
  \[
  H_0^W(\edub{W};\bfK_{\IZ}) \to K_0(\IZ W)
  \]
  is an isomorphism;

\item\label{lem:Projective_class_group_for_Coxeter_groups:rational}
  We have $\widetilde{K}_0(\IZ W) \otimes_{\IZ} \IZ[1/2] = \{0\}$.

 \end{enumerate}
\end{theorem}
\begin{proof}~\ref{lem:Projective_class_group_for_Coxeter_groups:eub(W)}
  We have for every group $G$ the obvious splitting
  $K_0(\IZ G) \cong K_0(\IZ) \oplus \widetilde{K}_0(\IZ G)$. 
 By \cite[Theorem~12.9]{Wall(1974a)}, $\widetilde K_0(\IZ[C_2^k])\cong \bigoplus_{i=3}^k{k\choose i}\IZ/(2^{i-2})$. This implies that for $H_*^?(-;\bfK_{\IZ})$ the groups $S_\sigma$ in Theorem~\ref{the:main_theorem} are given by $\IZ/(2^{\dim \sigma-1})$.
  Now the assertion follows from
  Theorem~\ref{the:main_theorem} applied to the equivariant homology theory $H_*^?(-;\bfK_{\IZ})$.
  \\[1mm]~\ref{lem:Projective_class_group_for_Coxeter_groups:Nil}
  This follows from Theorem~\ref{the:passage_from_calfin_to_calvycy}~%
\ref{the:passage_from_calfin_to_calvycy:special_K-theory}.
\\[1mm]~\ref{lem:Projective_class_group_for_Coxeter_groups:FJC}
This follows from the fact that a right-angled Coxeter group satisfies the Farrell--Jones Conjecture.
\\[1mm]~\ref{lem:Projective_class_group_for_Coxeter_groups:rational}
This follows from assertions~\ref{lem:Projective_class_group_for_Coxeter_groups:eub(W)},~%
\ref{lem:Projective_class_group_for_Coxeter_groups:Nil},
and~\ref{lem:Projective_class_group_for_Coxeter_groups:FJC}.
\end{proof}


\subsection{Whitehead group}
\label{subsec:Whitehead_group}

\begin{theorem}\label{the:Whitehead_group_of_Coxeter_groups}\
  \begin{enumerate}
  \item\label{the:Whitehead_group_of_Coxeter_groups:eub(W)}
  The canonical map
  \[H_1^W(EW;\bfK_{\IZ}) \to H_1^W(\eub{W};\bfK_{\IZ})
  \]
  is an isomorphism and we have
  an isomorphism
    \[H_1^W(EW;\bfK_{\IZ})\cong H_1(W;\IZ) \oplus K_1(\IZ) ;
     \]

\item\label{lem:Whitehead_group_of_Coxeter_groups:Nil}
The map
\[H_1^W(\eub{W};\bfK_{\IZ}) \to H_1^W(\edub{W};\bfK_{\IZ})
\]
is an isomorphism after inverting $2$;

\item\label{lem:Whitehead_group_of_Coxeter_groups:FJC}
  The canonical map
  \[
  H_1^W(\edub{W};\bfK_{\IZ}) \to K_1(\IZ W)
  \]
  is an isomorphism;

\item\label{lem:Whitehead_group_of_Coxeter_groups:FJC:after_inverting_2}
  We have $K_1(\IZ W) \otimes_{\IZ} \IZ[1/2] = \{0\}$.
\end{enumerate}
\end{theorem}
\begin{proof}\ref{the:Whitehead_group_of_Coxeter_groups:eub(W)}
  Notice that we have isomorphisms
  \begin{multline*}
  H_1^W(EW;\bfK_{\IZ}) \cong H_1(BW;\bfK(\IZ)) \cong H_1(BW,K_0(\IZ)) \oplus H_0(BW,K_1(\IZ))
  \\
  \cong H_1(W;\IZ) \oplus K_1(\IZ) \cong  W/[W,W] \oplus \{\pm 1\}.
\end{multline*}
  Hence it remains to show
  that the canonical map $H_1^W(EW;\bfK_{\IZ}) \to H_1^W(\eub{W};\bfK_{\IZ})$ is bijective.
  The Whitehead group $\Wh(C_2^k)$ vanishes for all natural numbers $k$
  by~\cite[Theorem~14.2~(iii) on page~330]{Oliver(1988)}. Hence the obvious map
  $H_1(C_2^k;\IZ) \times K_1(\IZ) \to K_1(\IZ[C_2^k])$ is an isomorphism.
  Now apply Theorem~\ref{the:main_theorem} to the equivariant homology theories given by the Borel
  construction, see Example~\ref{exa:Borel_homology}, and
  to $H^?_*(-;\bfK_{\IZ})$.
  \\[1mm]~\ref{lem:Whitehead_group_of_Coxeter_groups:Nil} This follows from
  Theorem~\ref{the:passage_from_calfin_to_calvycy}~\ref{the:passage_from_calfin_to_calvycy:special_K-theory}.
  \\[1mm]~\ref{lem:Whitehead_group_of_Coxeter_groups:FJC} This follows from the fact that
  a right-angled Coxeter group satisfies the Farrell--Jones Conjecture.
  \\[1mm]~\ref{lem:Whitehead_group_of_Coxeter_groups:FJC:after_inverting_2} This follows
  from Theorem~\ref{the:group_homology_Coxeter_groups} and
  assertions~\ref{the:Whitehead_group_of_Coxeter_groups:eub(W)},~%
\ref{lem:Whitehead_group_of_Coxeter_groups:Nil},
  and~\ref{lem:Whitehead_group_of_Coxeter_groups:FJC}.
\end{proof}


\subsection{Rationalized $K$-groups}
\label{subsec:Rationalized_K_groups_for_Coxeter_groups}

Let $R$ be a ring.
For any non-empty simplex $\sigma$ of $\Sigma$ we
have the diagonal embedding
  \[
   \Delta_{\sigma} \colon C_2 \to W(\sigma) = \prod_{v \in V \cap \sigma} W_v = \prod_{v \in V \cap \sigma} C_2.
 \]
 Let $j_{\sigma} \colon W(\sigma) \to W$ be the inclusion. Then
 $j_{\sigma} \circ \Delta_{\sigma} \colon C_2 \to W$ induces a homomorphisms
 $(j_{\sigma} \circ \Delta_{\sigma})_* \colon K_n(R[C_2]) \to
 K_n(RW)$. Denote by
 \[i_{\sigma,n} \colon \ker\bigl(K_n(R[C_2]) \to K_n(R)\bigr) \to K_n(RW)
 \]
 its composite with the inclusion
 $\ker\bigl(K_n(R[C_2]) \to K_n(R)\bigr) \to K_n(R[C_2])$, where
 $K_n(R[C_2]) \to K_n(R)$ is the homomorphism  induced by the projection $C_2 \to \{1\}$. Let
 $i_{\emptyset,n} \colon K_n(R) \to K_n(RW)$ be the map induced by the inclusion
 $\{1\} \to W$.

\begin{theorem}
	\label{the:rationalized_K_n(ZW)_Coxeter}
	Let $R$ be a regular ring.
  \begin{enumerate}
  \item\label{the:rationalized_K_n(ZW)_Coxeter:rational_iso}
    The map 
    \[
      i_{\emptyset,n} \oplus \bigoplus_{\substack{\sigma \in \calp \\\sigma \not= \emptyset}}
      i_{\sigma,n} \colon
      K_n(R)  \oplus \bigoplus_{\substack{\sigma \in \calp \\\sigma \not= \emptyset}} \ker\bigl(K_n(R[C_2])
      \to K_n(R)\bigr)  \to K_n(RW)
    \]
    is rationally an isomorphism for all $n \in \IZ$;
	
    \item\label{the:rationalized_K_n(ZW)_Coxeter:rank}
 We have for $R=\IZ$
  \[
  \IQ \otimes_{\IZ} K_n(\IZ W) \cong
  \begin{cases}
  \IQ^{r} & \mbox{if} \; n = 4k+1 \; \mbox{with} \; k \ge 1; \\
  \IQ & \mbox{if} \; n = 0;\\
  \{0\} & \mbox{otherwise.}
\end{cases}
\]
\end{enumerate}
\end{theorem}
\begin{proof}~\ref{the:rationalized_K_n(ZW)_Coxeter:rational_iso}
	Notice that any non-trivial finite cyclic subgroup $C$ of $C_2^k$ is isomorphic to $C_2$ and that the obvious composite
	\begin{multline*}
  \ker\bigl(\IQ \otimes_{\IZ}  K_n(R C) \to \IQ \otimes_{\IZ}  K_n(R)\bigr)
 \to \IQ \otimes_{\IZ}  K_n(R C) 
 \\
 \to \cok\bigl(\IQ \otimes_{\IZ} K_n(R)   \to \IQ \otimes_{\IZ}  K_n(R C)\bigr)
\end{multline*}
	is an isomorphism.
	The isomorphism
  appearing in~\cite[(2.11)]{Bartels-Lueck(2007ind)}, which exists for $\Lambda = \IQ$ and
  the equivariant homology theory $H_*^{?}(-;\bfK_R)$
  because of~\cite[Lemma~4.1~(e)]{Bartels-Lueck(2007ind)}, which in turn follows from~\cite[Corollary 4.2]{Swan(1960a)},
  boils down to an isomorphism
  \begin{multline*}
     j_{\{1\}} \oplus \bigoplus_{\substack{C \subseteq C_2^k\\C \cong C_2}} j_C \colon
    \IQ \otimes_{\IZ}  K_n(R) \oplus \bigoplus_{\substack{C \subseteq C_2^k\\C \cong C_2}}
    \ker\bigl(\varepsilon_C^\IQ\colon\IQ \otimes_{\IZ}  K_n(RC) \to  \IQ \otimes_{\IZ} K_n(R)\bigr)
    \\
    \xrightarrow{\cong} \IQ \otimes_{\IZ} K_n(R[C_2^k]),
  \end{multline*}
  where $\varepsilon_C^\IQ\colon \IQ \otimes_{\IZ} K_n(RC) \to \IQ \otimes_{\IZ} K_n(R)$
  is induced by the  projection $C \to \{1\}$, the map $j_{\{1\}}$ is induced by the inclusion $\{1\} \to C$,
    and $j_C$ is the composite of the inclusion
  $\ker( \varepsilon_C^\IQ) \to \IQ
  \otimes_{\IZ} K_n(RC)$ with the map
  $\IQ \otimes_{\IZ} K_n(RC) \to \IQ \otimes_{\IZ} K_n(R[C_2^k])$ coming from the
  inclusion $C \to C_2^k$.
  By naturality we get a commutative diagram
  \[
  \xymatrix{\bigoplus\limits_{i=1}^k \left(\IQ \otimes_{\IZ}  K_n(R) \oplus \bigoplus\limits_{\substack{C \subseteq (C_2^k)_i\\C \cong C_2}}
    \ker( \varepsilon_C^\IQ)\right)
    \ar[r]^-{\cong} \ar[d]
    &\bigoplus\limits_{i=1}^k \IQ \otimes_{\IZ} K_n(R[(C_2^k)_i]) \ar[d]
    \\
    \IQ \otimes_{\IZ}  K_n(R) \oplus \bigoplus\limits_{\substack{C \subseteq C_2^k\\C \cong C_2}}
    \ker( \varepsilon_C^\IQ)
    \ar[r]_-{\cong}
    &\IQ \otimes_{\IZ} K_n(R[C_2^k])}
\]
where the vertical arrows come from the inclusions $(C_2^k)_i \to C_2^k$.  Notice that a
cyclic subgroup of $C_2^k$ belongs to $(C_2^k)_i$ for some $i \in \{1,2, \ldots, k\}$ if
and only if it is different from the diagonal subgroup
$(C_2^k)_{\Delta} \coloneqq \{(a,a, \ldots, a) \mid a \in C_2\} \subseteq C_2^k$.  Hence the
composite
\begin{multline*}
  \ker\left(K_n(R [(C_2^k)_{\Delta}]) \to  K_n(R)\right)
\to K_n(R [(C_2^k)_{\Delta}]) \to K_n(R [C_2^k])
\\
\to \cok\bigl(\bigoplus_{i=1}^k K_n(R[(C_2^k)_i]) \to K_n(R[C_2^k])\bigr)
\end{multline*}
is rationally bijective.

Now assertion~\ref{the:rationalized_K_n(ZW)_Coxeter:rational_iso} follows from
Theorem~\ref{the:main_theorem} and
Theorem~\ref{the:passage_from_calfin_to_calvycy}~\ref{the:passage_from_calfin_to_calvycy:special_K-theory}.
\\[1mm]~\ref{the:rationalized_K_n(ZW)_Coxeter:rank}
Due to Borel~\cite{Borel(1974)} we know for $R=\IZ$ that
\[
  \IQ\otimes_{\IZ} K_n(\IZ)  \cong
  \begin{cases}
  \IQ & \mbox{if} \; n = 4k+1 \;\mbox{with} \; k \ge 1; \\
  \IQ & \mbox{if} \; n = 0; \\
  \{0\} & \mbox{otherwise.}\end{cases}
\]
We get from~\cite[Theorem~2.2]{Jahren(2009)} for $C \cong C_2$
\[
  \IQ\otimes_{\IZ} K_n(\IZ[C])  \cong
  \begin{cases}
  \IQ^2  & \mbox{if} \; n = 4k+1 \;\mbox{with} \; k \ge 1; \\
  \IQ& \mbox{if} \; n = 0; \\
  \{0\} & \mbox{otherwise.}\end{cases}
\]
Hence we get
\[
  \ker\bigl( \IQ \otimes_{\IZ} K_n(\IZ C) \to \IQ \otimes_{\IZ} K_n(\IZ)\bigr)\cong
  \begin{cases}
    \IQ  & \mbox{if} \; n = 4k+1 \;\mbox{with} \; k \ge 1; \\
  \{0\} & \mbox{otherwise.}\end{cases}
\]
Now assertion~\ref{the:rationalized_K_n(ZW)_Coxeter:rank} follows from
assertion~\ref{the:rationalized_K_n(ZW)_Coxeter:rational_iso}.  
\end{proof}


\subsection{$L$-groups  after inverting $2$}
\label{subsec:L-groups_after_inverting_2}

The maps appearing in the result below are defined analogously to the maps appearing in
Theorem~\ref{the:rationalized_K_n(ZW)_Coxeter}.

\begin{theorem}
\label{the:Algebraic_l-theory_for_rings_R_after_inverting_2_Coxeter}
  Let $R$ be a ring with involution. 
  \begin{enumerate}
  \item\label{the:Algebraic_l-theory_for_rings_R_after_inverting_2_Coxeter:iso}
  The map 
    \[
      i_{\emptyset,n} \oplus \bigoplus_{\substack{\sigma \in \calp\\\sigma \not= \emptyset}}
      i_{\sigma,n} \colon
      L_n^{\langle-\infty \rangle}(R)  \oplus \bigoplus_{\substack{\sigma \in \calp\\\sigma \not= \emptyset}}
      \ker\bigl(L_n^{\langle-\infty \rangle}(R[C_2])
      \to L_n^{\langle-\infty \rangle}(R)\bigr)  \to L_n^{\langle-\infty \rangle}(RW)
    \]
    is an isomorphism after inverting 2 for all $n\in\IZ$; 
  \item\label{the:Algebraic_l-theory_for_rings_R_after_inverting_2_Coxeter:rank}
  We have for $R\in\{\IZ,\IQ,\IR\}$
  \[
  L_n^{\langle-\infty \rangle}(R W)[1/2] \cong
  \begin{cases}
  \IZ[1/2]^{r} & \mbox{if} \; n = 4k \; \mbox{for} \; k \in \IZ; \\
  \{0\} & \mbox{otherwise.}
  \end{cases}
  \]
  \end{enumerate}
\end{theorem}

\begin{proof}~\ref{the:Algebraic_l-theory_for_rings_R_after_inverting_2_Coxeter:iso}
  Note  that any non-trivial subgroup of the form $C \times P$ of $C_2^k$ for a cyclic group $C$ and a
  $p$-group $P$ for an odd prime number $p$ is isomorphic to $C_2$. 
  The isomorphism
  appearing in~\cite[(2.11)]{Bartels-Lueck(2007ind)} which exists for $\Lambda = \IZ[1/2]$ and
  the equivariant homology theory $H_*^{?}(-;\bfL^{\langle - \infty \rangle}_R)$
  because of~\cite[Theorem~2]{Dress(1975)}, boils down to an isomorphism 
  \begin{multline*}
    j_{\{1\}}  \oplus \bigoplus_{\substack{C \subseteq C_2^k\\C \cong C_2}} j_C \colon
    L_n^{\langle -\infty \rangle}(R)[1/2]  \oplus \bigoplus_{\substack{C \subseteq C_2^k\\C \cong C_2}}
    \ker\bigl(L_n^{\langle -\infty \rangle} (RC)[1/2] \to  L_n^{\langle -\infty \rangle} (R)[1/2]\bigr)
    \\
    \xrightarrow{\cong} L_n^{\langle -\infty \rangle} (R[C_2^k])[1/2],
\end{multline*}
  where the map $j_{\{1\}}$ is induced by the inclusion $\{1\} \to C$, and $j_C$ is the composite of the inclusion
  of $\ker\bigl(L_n^{\langle -\infty \rangle} (RC)[1/2] \to  L_n^{\langle -\infty \rangle} (R)[1/2]\bigr)$
  into $L_n^{\langle -\infty \rangle} (RC)[1/2]$ with the map
  $L_n^{\langle -\infty \rangle} (RC)[1/2]\to L_n^{\langle -\infty \rangle} (R[C_2^k])[1/2]$ coming from the
  inclusion $C \to C_2^k$.  Now
  Theorem~\ref{the:Algebraic_l-theory_for_rings_R_after_inverting_2_Coxeter} follows completely analogous to the argument appearing in the proof of
  Theorem~\ref{the:rationalized_K_n(ZW)_Coxeter}~\ref{the:rationalized_K_n(ZW)_Coxeter:rational_iso}.
  \\[1mm]~\ref{the:Algebraic_l-theory_for_rings_R_after_inverting_2_Coxeter:rank}
This follows from
assertion~\ref{the:Algebraic_l-theory_for_rings_R_after_inverting_2_Coxeter:iso} using~\cite[Proposition~22.34 on page~254]{Ranicki(1992)}.
\end{proof}


\subsection{$K$- and $L$-groups for $R$ containing $1/2$}
\label{subsec:K-and_L-grouos_for_R_containing_1/2}

\begin{theorem}
\label{the:Algebraic_K-_and_L-theory_for_rings_R_containing_1/2}

  For all $n \in \IZ$ there are explicit isomorphisms
  \begin{enumerate}
  	\item 
  	$\bigoplus_{\sigma \in \calp} K_n(R)
  	\xrightarrow{\cong} 
  	K_n(RW)$ if $R$ is regular and contains $1/2$;
  	\item $\bigoplus_{\sigma \in \calp} L_n^{\langle -\infty \rangle}(R)
  	 \xrightarrow{\cong} 
  	L_n^{\langle -\infty \rangle}(RW)$ if $R$ contains $1/2$.
  \end{enumerate}
\end{theorem}

Its proof needs some preparations.
In the sequel we
will write $C_2^k$ multiplicatively and we denote by $t_i$ the generator of the $i$-th factor $C_2$
viewed as an element in $C_2^k$ for $i = 1,2, \ldots, k$.

Let $R$ be a ring in which $2$ is invertible.
We get a decomposition of rings, natural in $R$,
  \[R[C_2] \xrightarrow{\cong} R \times R, \quad a +bt \mapsto (a+b,a -b).
  \]
  Its inverse sends $(c,d)$ to $~\frac{1}{2} \cdot \bigl((c+d) + (c-d) \cdot t \bigr)$. Since algebraic $K$-theory is
  compatible with products, we obtain an isomorphism, natural in $R$,
  \[
  S_n(R) \colon K_n(R[C_2]) \xrightarrow{\cong} K_n(R) \times K_n(R).
\]
One can iterate this using the obvious ring isomorphism $(R[C_2^{k-1}])[C_2] \cong R[C_2^k]$
and thus obtains an isomorphism
\begin{equation}
S_n^k(R) \colon K_n(R[C_2^k]) \xrightarrow{\cong} \prod_{\epsilon \in \hom_{\IZ}(C_2^k,\{\pm 1\})} K_n(R),
\label{iso_S_n_upper_k}
\end{equation}
which comes from the isomorphism of rings
\[
 R[C_2^k] \xrightarrow{\cong} \prod_{\epsilon \in \hom_{\IZ}(C_2^k,\{\pm 1\})} R, \quad 
  \sum_{g \in C_2^k} \lambda_g \cdot g
\mapsto \bigl(\sum_{g \in C_2^k} \lambda_g \cdot \epsilon(g)\bigr)_{\epsilon}.
\]
Its inverse is given by
\[(\mu_{\epsilon})_{\epsilon}
  \mapsto 2^{-k} \cdot \sum_{g \in C_2^k}
  \bigl(\sum_{\epsilon} \epsilon(g) \cdot \mu_{\epsilon}\bigr) \cdot g.
\]

\begin{lemma}\label{lem:Cokernel_relevant_for_Coxeter_groups}

  \begin{enumerate}
   \item\label{lem:Cokernel_relevant_for_Coxeter_groups:alg_K-theory}
  Suppose that $2$ is invertible in $R$.  Then there is an isomorphism
  \[
    d_n^k \colon \cok\left(\bigoplus_{i = 1}^k K_n(R[j_i^k]) \colon
      \bigoplus_{i = 1}^k K_n(R[(C_2^k)_i])  \to  K_n(R[C_2^k])\right)
    \xrightarrow{\cong} K_n(R).
  \]
  Its inverse is the composite of the homomorphism
  \[
    \beta \colon K_n(R) \to K_n(R[C_2^k])
  \]
  coming from the ring homomorphism   $R \to R[C_2^k]$ sending $\lambda$ to
  $2^{-k} \cdot \lambda \cdot \prod_{i=1}^k (1-t_i)$ with the projection
  $K_n(R[C_2^k])\to \cok\left(\bigoplus_{i = 1}^k K_n(R[j_i^k])\right)$.
  The homomorphism
  $\beta$ agrees with the restriction of the inverse of the isomorphism
  $S_n^k(R)$ of~\eqref{iso_S_n_upper_k} to the factor $K_n(R)$ which belongs to
  $\epsilon$ given by $\epsilon(t_i) = -1$ for $i = 1,2 ,\ldots k$;
  
  \item\label{lem:Cokernel_relevant_for_Coxeter_groups:alg_L-theory}
  The same assertion holds if we replace algebraic $K$-theory by algebraic $L$-theory
  with the decoration $\langle -\infty \rangle$;

  \item\label{lem:Cokernel_relevant_for_Coxeter_groups:alg_L-theory_top_K-theory}
    The same assertion is true if we take $R$ to be $\IR$ or $\IC$ and we replace algebraic
  $K$-theory by topological $K$-theory.
\end{enumerate}

\end{lemma}
\begin{proof} We give the proof for algebraic $K$-theory only, the one for the other cases
  is completely analogous.

  We use induction over $k$. If $k = 0$, the map $d^0_n$ comes from the identification
  \[
  K_n(R[C_2^0]) = K_n(R[\{1\}]) = K_n(R).
   \]
The induction step from $(k-1)$ to $k \ge 1$ is done as follows.
We have the following commutative diagram of $\IZ$-modules
  \[
    \xymatrix@!C=18em{\bigoplus_{i =1}^k K_n(R[(C_2^k)_i])
      \ar[r]^-{\bigoplus_{i =1}^k K_n(j^k_i)} \ar[d]_{\id}^{\cong} &
      K_n(R[C_2^k]) \ar[dd]^{\id}_{\cong}
      \\
      K_n(R[(C_2^k)_k]) \oplus \bigoplus_{i =1}^{k-1}
      K_n(R[(C_2^k)_i]) \ar[d]_{\id \oplus \bigoplus_{i
          =1}^{k-1} \id}^{\cong} &
      \\
      K_n(R[C_2^{k-1}]) \oplus \bigoplus_{i =1}^{k-1}
      K_n(R[(C_2^{k-1})_i][C_2]) \ar[d]_-{\id \oplus
        \bigoplus_{i =1}^{k-1} S_n(R[(C_2^{k-1})_i])}^-{\cong} &
      K_n(R[C_2^{k-1}][C_2])
      \ar[dd]^-{S_n(R[C_2^{k-1}])}_-{\cong}
      \\
      {\begin{array}{c}
      K_n(R[C_2^{k-1}]) 
      \\
      \oplus 
      \\
      \bigoplus_{i =1}^{k-1}
      \bigl(K_n(R[(C_2^{k-1})_i]) \oplus
      K_n(R[(C_2^{k-1})_i])\bigr) 
	  \end{array}}      
      \ar[d]_-{f}^-{\cong} &
      \\
      {\begin{array}{c} K_n(R[C_2^{k-1}])
         \\
         \oplus \\\bigoplus_{i =1}^{k-1} K_n(R[(C_2^{k-1})_i])
         \\
         \oplus
         \\
         \bigoplus_{i =1}^{k-1} K_n(R[(C_2^{k-1})_i])
       \end{array}}
     \ar[r]_-{\begin{pmatrix} \id &    u   & 0\\ \id &  0 & u \\ \end{pmatrix}}
     &
     {\begin{array}{c}
        K_n(R[C_2^{k-1}])
        \\
        \oplus
        \\K_n(R[C_2^{k-1}])
      \end{array}}
  }
\]
where $f$ is the obvious isomorphism and
$u \coloneqq \bigoplus_{i =1}^{k-1}K_n( j_i^{k-1})$. The diagram above induces
an isomorphism
\[
  \cok\left(\bigoplus_{i =1}^k K_n(j^k_i)\right)
\xrightarrow{\cong}
\cok\begin{pmatrix} \id &    u   & 0\\ \id &  0 & u \\ \end{pmatrix}.
\]
If $k = 1$, then $\begin{pmatrix} \id &    u   & 0\\ \id &  0 & u \\ \end{pmatrix}$ reduces to
$K_n(R) \to K_n(R) \oplus K_n(R), \; x \mapsto (x,x)$
and the desired isomorphism $d_n^1$ is induced by
$K_n(R) \oplus K_n(R) \to K_n(R), (x,y) \mapsto x-y$.
Suppose $k \ge 2$.
Since the first and third map in the composite 
\[
  \begin{pmatrix} \id & 0 \\\ -\id & \id \end{pmatrix} \cdot
  \begin{pmatrix} \id &    u   & 0\\ \id &  0 & u \\ \end{pmatrix} \cdot
    \begin{pmatrix} \id & 0 & 0 \\ 0 & \id & 0 \\ 0 & \id & \id \end{pmatrix}
  \]
are isomorphisms and the composite is given by
$\begin{pmatrix} \id & u & 0 \\ 0 & 0 & u\end{pmatrix}$,
we obtain an isomorphism
\[\cok\begin{pmatrix} \id &    u   & 0\\ \id &  0 & u \\ \end{pmatrix}
  \xrightarrow{\cong}
  \cok\begin{pmatrix} \id &    u   & 0\\ 0 &  0 & u \\ \end{pmatrix}.
\]
Since $\begin{pmatrix} \id &    u  \end{pmatrix} \colon
K_n(R[C_2^{k-1}]) \oplus \bigoplus_{i =1}^{k-1} K_n(R[(C_2^{k-1})_i])
\to K_n(R[C_2^{k-1}])$
is surjective,
we obtain an isomorphism
\[ \cok\begin{pmatrix} \id &    u   & 0\\ 0 &  0 & u \\ \end{pmatrix}
\xrightarrow{\cong}
\cok(u).
\]
Its inverse is induced by the composite 
\[
K_n(R[C_2^{k-1}]) \xrightarrow{\begin{pmatrix} 0 \\ \id \end{pmatrix}}
  K_n(R[C_2^{k-1}])  \oplus K_n(R[C_2^{k-1}])\;
  \xrightarrow{S_n(R[C_2^{k-1}])^{-1}} K_n(R[C_2^k]),
\]
which is the homomorphism $K_n(R[C_2^{k-1}]) \to K_n(R[C_2^{k}])$
induced by the ring homomorphism $R[C_2^{k-1}] \to R[C_2^{k}]$
sending $x$ to $\frac{1}{2} \cdot x \cdot (1-t_k)$.  Since the
induction hypothesis applies to $u$,
Lemma~\ref{lem:Cokernel_relevant_for_Coxeter_groups} follows.
\end{proof}

Now Theorem~\ref{the:Algebraic_K-_and_L-theory_for_rings_R_containing_1/2} follows from
Theorem~\ref{the:main_theorem},
Theorem~\ref{the:passage_from_calfin_to_calvycy}~%
\ref{the:passage_from_calfin_to_calvycy:general_K-theory_R_reg_Q_subseteq_R}
and~\ref{the:passage_from_calfin_to_calvycy:general_L-theory_1/2_in_R},
and Lemma~\ref{lem:Cokernel_relevant_for_Coxeter_groups}.


\subsection{Topological $K$-theory}
\label{subsec:Topological_K-theory}

\begin{theorem}
	\label{the:Topological-K-theory_Coxeter}
  There are for every $n \in \IZ$ isomorphisms
  \begin{eqnarray*}
    \bigoplus_{\sigma \in \calp} K_n(\IC)
    & \xrightarrow{\cong} &
   K_n(C^*_m(W)) \cong K_n(C^*_r(W));
    \\
    \bigoplus_{\sigma \in \calp} KO_n(\IR)
    & \xrightarrow{\cong} &
    KO_n(C^*_m(W;\IR)) \cong KO_n(C^*_r(W;\IR)).
  \end{eqnarray*}                       

  In particular there are isomorphisms of abelian groups
  \begin{eqnarray*}
    K_n(C^*_m(W)) \cong K_n(C^*_r(W))\cong
    \begin{cases}
      \IZ^r & \text{if}\; n \; \text{is even};
      \\
      \{0\}  & \text{otherwise};
    \end{cases}
    \\
    KO_n(C^*_m(W;\IR)) \cong KO_n(C^*_r(W;\IR))\cong
    \begin{cases}
      \IZ^r & \text{if}\; n \equiv 0 \mod 4;
      \\
      (\IZ/2)^r & \text{if}\; n \equiv 1,2 \mod 8;
      \\
      \{0\}  & \text{otherwise}.
    \end{cases}
  \end{eqnarray*}
\end{theorem}
\begin{proof}
 This follows from Theorem~\ref{the:main_theorem} and Lemma~\ref{lem:Cokernel_relevant_for_Coxeter_groups}.
\end{proof}

The result for complex coefficients was already obtained by S\'anchez-Garc\'ia using the Davis complex as a model for $\eub{W}$ in~\cite{Sanchez-Garcia(2006Cox)}. 
In special cases, the topological $K$-theory of $\eub{W}$ was computed in Fuentes Rum\'i's masters's thesis~\cite{Fuentes-Rumi(2018)}.

\begin{remark}
	\label{rem:Explicite_isomorphism}
  In Subsection~\ref{subsec:Computations} we have given an explicit description of
  the isomorphism
  $\bigoplus_{\sigma \in \calp} K_n(\IC) \xrightarrow{\cong} K_n(C^*_r(G))$ above
  which actually carries over to many of the other situations. In order to prove the description, one
  has to go through the construction of the isomorphism and to make in the application of
  Lemma~\ref{lem:widehat(s)_bijective} the right choice for $\widehat{s}_{\tau}$. Namely,
  one takes for $\widehat{s}_{\tau}$ the composite of the homomorphism $\beta$ with the isomorphism
  $d_n^k$ appearing in
  assertion~\ref{lem:Cokernel_relevant_for_Coxeter_groups:alg_K-theory} of
  Lemma~\ref{lem:Cokernel_relevant_for_Coxeter_groups}.
\end{remark}


\typeout{-------------  An example ----------------------}

\section{An example}
\label{sec:An_example}

In this section, we want to apply the computations from the previous sections to a concrete example. For this we picked the group $W\coloneqq\IZ/2\times\IZ/2\times D_\infty\times\IZ$. Note that it is a graph product with vertex groups $\IZ/2$ and $\IZ$. In~\cite[Example~3.28]{Davis-Khan-Ranicki(2011)} Davis, Khan and Ranicki showed that the Whitehead group of $W$ is infinitely generated due to Nil elements.

It will be useful to consider $W$ as $W_0\times\IZ$, where $W_0=\IZ/2\times\IZ/2\times D_\infty$ is the right-angled Coxeter group associated to the simplicial graph $X$ with vertex set $V=\{1,2,3,4\}$ whose edges are $\{1,2\}$, $\{2,3\}$, $\{3,4\}$, $\{1,4\}$, and $\{1,3\}$.  Then the flag complex $\Sigma$ associated to $X$ is the suspension of a one-simplex so that in the notation of Section~\ref{sec:right-angled_Coxeter_groups} we have $r_{-1}=1$, $r_0=4$, $r_1=5$, $r_2=2$, and $r=12$.

We conclude from Theorem~\ref{the:group_homology_Coxeter_groups} for $n \ge 2$
\begin{align*}
H_n(W;\IZ) &\cong H_n(W_0;\IZ) \oplus H_{n-1}(W_0;\IZ) \cong \bigoplus_{i=1}^{u_n}C_2;\\
H_1(W;\IZ) &\cong \IZ \oplus \bigoplus_{i=1}^4 C_2, 
\end{align*}
where
\begin{align*}
u_n 
&= \sum_{k=0}^2 r_k \cdot \rho_{n,k} + \sum_{k=0}^2 r_k \cdot \rho_{n-1,k}\\
&= 4\cdot (\rho_{n,0}+\rho_{n-1,0}) + 5\cdot (\rho_{n,1}+\rho_{n-1,1}) + 2\cdot (\rho_{n,2}+\rho_{n-1,2})\\
&= \begin{cases}
4 + 5(k+k-1) + 2(k(k-1)+(k-1)^2) &\text{if}\; n=2k\; \text{for}\; k\ge1; \\
4 + 5(k+k) + 2(k^2+k(k-1)) &\text{if}\; n=2k+1\; \text{for}\; k\ge 1;
\end{cases}\\
&= \begin{cases}
4k^2+4k+1 &\text{if}\; n=2k\; \text{for}\; k\ge1; \\
4k^2+8k+4 &\text{if}\; n=2k+1\; \text{for}\; k\ge 1.
\end{cases}
\end{align*}

Note that the group $W$ satisfies the Baum--Connes Conjecture and the Farrell--Jones Conjecture since it is a graph product of abelian groups. Hence for every regular ring $R$ the assembly map
\[
H_n^W(\eub{W};\bfK_R) \to K_n(R[W])
\]
is bijective after inverting $2$ by Theorem~\ref{the:passage_from_calfin_to_calvycy} \ref{the:passage_from_calfin_to_calvycy:special_K-theory}.

The proof of Theorem~\ref{the:negative_alg_K-theory_of_Coxeter_groups} applies verbatim to the group $W$ so that we obtain
\[
K_n(\IZ W) = \{0\} \quad \text{for}\; n \le -1.
\]

For any equivariant homology theory we have
\begin{equation}
\label{eq:decomposition_W}
\calh_n^W(\eub W)\cong \calh_n^{W_0}(\eub W_0\times S^1)\cong \calh_n^{W_0}(\eub W_0)\oplus\calh_{n-1}^{W_0}(\eub W_0).
\end{equation}

Using~\eqref{eq:decomposition_W}, we have
\[\IZ[1/2]\otimes_\IZ \widetilde K_0(\IZ W)=\{0\}\]
by Theorem~\ref{the:projective_class_groups_Coxeter_groups} and Theorem~\ref{the:negative_alg_K-theory_of_Coxeter_groups}.

By \eqref{eq:decomposition_W}, Theorem~\ref{the:projective_class_groups_Coxeter_groups} and Theorem~\ref{the:Whitehead_group_of_Coxeter_groups}, we have 
\[
\IZ[1/2] \otimes_{\IZ} K_1(\IZ W)\cong \IZ[1/2]\quad\text{and}\quad\IZ[1/2] \otimes_{\IZ} \Wh(W) = \{0\}.
\]
Note that without inverting two, the Whitehead group $\Wh(W)$ contains a non-trivial Nil term by~\cite[Example~3.28]{Davis-Khan-Ranicki(2011)} as mentioned above.

By \eqref{eq:decomposition_W} and  Theorem~\ref{the:rationalized_K_n(ZW)_Coxeter}
\begin{equation*}
\IQ \otimes_{\IZ} K_n(\IZ W) \cong
\begin{cases}
\IQ^{12} & \mbox{if} \; n = 4k+1 \; \mbox{or} \; n = 4k+2  \; \mbox{for} \; k \ge 1; \\
\IQ & \mbox{if} \; n = 0,1;\\
\{0\} & \mbox{otherwise.}
\end{cases}
\end{equation*}

The Shaneson splitting yields for all $n\in\IZ$ an isomorphism
\[
L_n^{\langle-\infty \rangle}(\IZ W) \cong L_n^{\langle-\infty\rangle}(\IZ W_0) \oplus L_{n-1}^{\langle-\infty\rangle}(\IZ W_0).
\]
Hence by Theorem~\ref{the:Algebraic_l-theory_for_rings_R_after_inverting_2_Coxeter} we find
\[
\IZ[1/2] \otimes_{\IZ} L_n^{\langle-\infty \rangle}(\IZ W) \cong
\begin{cases}
\IZ[1/2]^{12} & \text{if} \; n = 4k \; \text{or} \; n = 4k+1 \; \text{for} \; k \in \IZ; \\
\{0\} & \text{otherwise}.
\end{cases}
\]

By \eqref{eq:decomposition_W}, we get from Theorem~\ref{the:Topological-K-theory_Coxeter} for all $n\in\IZ$
\begin{equation*}
K_n(C^*_r(W)) \cong \IZ^{12}.
\end{equation*}

\addcontentsline{toc<<}{section}{References} 
\bibliographystyle{abbrv}
\bibliography{dbpub,dbpre}

\end{document}